\documentclass[11pt]{amsart}
\usepackage[T1]{fontenc}

\usepackage{amsmath,amsfonts,amsbsy,amsgen,amscd,mathrsfs,amssymb,amsthm}
\usepackage{enumerate,mathtools}

\usepackage{xcolor}
\usepackage{hyperref}

\hypersetup{
  colorlinks=true,
  linkcolor=red,
  citecolor=blue,
  urlcolor=black
}

\makeatletter
\renewcommand{\eqref}[1]{%
  \hyperref[#1]{\textup{\tagform@{\ref*{#1}}}}%
}
\makeatother

\usepackage{subcaption}

\usepackage{xcolor}

\usepackage{amssymb}
\usepackage{graphics}
\usepackage[mathcal]{euscript}
\usepackage{indentfirst}
\usepackage{braket}
\usepackage{amssymb}
\usepackage{latexsym}
\usepackage{amsmath}
\usepackage{amsthm}
\usepackage{amscd}
\usepackage{enumerate}
\usepackage{graphicx}
\usepackage{fullpage}
\usepackage{url}

\usepackage{mathrsfs}

\numberwithin{equation}{section}

\newtheorem{theorem}{Theorem}[section]
\newtheorem{proposition}[theorem]{Proposition}

\newtheorem{corollary}[theorem]{Corollary}
\newtheorem{lemma}[theorem]{Lemma}
\newtheorem{example}[theorem]{Example}
\theoremstyle{definition}
\newtheorem{definition}[theorem]{Definition}
\newtheorem{claim}[theorem]{Claim}

\newtheorem{remark}[theorem]{Remark}

\newtheorem*{claim*}{Claim}

\newcommand{\R}{\mathbb R}

\newcommand{\supp}{\mathop{\mathrm{supp}}}
\newcommand{\proj}{\mathrm{Proj}}

\newcommand{\Proj}{\operatorname{Proj}}
\newcommand{\Prox}{\operatorname{Prox}}
\newcommand{\argmin}{\operatorname{argmin}}
\newcommand*\diff{\mathop{}\!\mathrm{d}}

\newcommand{\scvx}{\alpha}

\newcommand{\prob}{\rho}
\newcommand{\tvec}{z}
\newcommand{\marg}{\mu}
\newcommand{\flow}{\pi}
\newcommand{\flowop}{\flow_{\textnormal{opt}}}
\newcommand{\m}{\mathrm m}
\newcommand{\mop}{\m_{\textnormal{opt}}}
\newcommand{\vel}{\mathrm v}
\newcommand{\velb}{\mathrm u}
\newcommand{\velop}{\vel_{\textnormal{opt}}}
\newcommand{\joint}{\gamma}
\newcommand{\jointop}{\joint_{\textnormal{opt}}}
\newcommand{\diffeo}{\Phi}
\newcommand{\reg}{\zeta}
\newcommand{\pspace}{\mathcal P}
\newcommand{\dd}{d}
\newcommand{\kk}{k}

\newcommand{\domainot}{\Omega}

\newcommand{\law}{\operatorname{Law}}
\newcommand{\test}{\eta}
\newcommand{\cost}{L}
\newcommand{\costt}{c}
\newcommand{\CC}{\mathcal C}
\newcommand{\actd}{\mathrm B}
\newcommand{\actdd}{\mathrm B}
\newcommand{\acts}{\Gamma}

\newcommand{\ID}{\mathrm D}

\newcommand{\base}{p}
\newcommand{\qm}{\nu}
\newcommand{\lag}{\mathcal L}
\newcommand{\lm}{\lambda}

\newcommand{\mapot}{\mathrm T_{\textnormal{opt}}}

\newcommand{\gridxc}{\mathcal G_x^{\textnormal{c}}}
\newcommand{\gridxs}{\mathcal G_x^{\textnormal{s}}}
\newcommand{\gridtc}{\mathcal G_t^{\textnormal{c}}}
\newcommand{\gridts}{\mathcal G_t^{\textnormal{s}}}
\newcommand{\fgridc}{\mathcal E^{\textnormal{c}}}
\newcommand{\fgrids}{\mathcal E^{\textnormal{s}}}

\newcommand{\Uc}{U^{\textnormal{c}}}
\newcommand{\Us}{U^{\textnormal{s}}}

\newcommand{\fc}{\flow^{\textnormal{c}}}
\newcommand{\fs}{\flow^{\textnormal{s}}}
\newcommand{\mc}{\m^{\textnormal{c}}}
\newcommand{\ms}{\m^{\textnormal{s}}}

\title{A dynamical formulation of multi-marginal optimal transport}

\author{Brendan Pass}
\address{Department of Mathematical and Statistical Sciences, 632 CAB, University of Alberta, Edmonton, Alberta}
\email{pass@ualberta.ca}

\author{Yair Shenfeld}
\address{Division of Applied Mathematics, Brown University, Providence, RI, USA}
\email{Yair\_Shenfeld@Brown.edu}

\begin{document}
\maketitle

\begin{abstract}

We present a primal-dual dynamical formulation of the multi-marginal optimal transport problem for (semi-)convex cost functions. Even in the two-marginal setting, this formulation applies to cost functions not covered by the classical dynamical approach of Benamou-Brenier. Our dynamical formulation yields a convex optimization problem, enabling the use of convex optimization tools to find quasi-Monge solutions of the static multi-marginal problem  for translation-invariant costs. We illustrate our results numerically with proximal splitting methods.
\end{abstract}

\section{Introduction}

\subsection{The two-marginal optimal transport problem}
\label{subsec:two_marg}
The Benamou-Brenier dynamical formulation of optimal transport   has had a tremendous impact on the field--both theoretically, in the form of the Otto calculus \cite[Chapter 8]{villani2021topics}, and numerically, by providing algorithms for computing optimal transport maps \cite[Chapter 6]{Santambrogio_book}. In contrast, an analogous dynamical formulation of the \emph{multi-marginal} optimal transport problem has not been found to date. (For different perspectives and special cases we refer to \cite{MR4208613, albergo2024multimarginal}.) To explain the challenge let us review the  Benamou-Brenier construction. The \emph{static} two-marginal optimal transport problem is to find a coupling $\joint$ of two given marginals $\marg_1,\marg_2$ on a domain $\domainot\subseteq\R^{\dd}$, which minimizes the cost
\begin{equation}
\label{eq:static_ot_intro}
\min_{\textnormal{ joint couplings $\joint$ of } \marg_1,\marg_2}\int_{\domainot^2}\costt(x_2-x_1)\diff \joint( x_1, x_2),
\end{equation}
where  $\costt:\R^{\dd}\to \R$ is a cost function. The problem \eqref{eq:static_ot_intro} is the Kantorovich formulation of the Monge optimal transport problem,
\begin{equation}
\label{eq:static_ot__monge_intro}
\min_{\textnormal{ transport maps $\diffeo$ between $\marg_1$ and $\marg_2$}}\int_{\domainot}\costt(\diffeo(x_1)-x_1)\diff\marg_1( x_1).
\end{equation}
The optimization problem \eqref{eq:static_ot_intro} is a relaxation of \eqref{eq:static_ot__monge_intro}, in the sense that couplings need not arise from transport maps. In turn, the Kantorovich formulation has the advantage of being a linear optimization problem since the objective in \eqref{eq:static_ot_intro} is linear in $\joint$, and the constraint that $\joint$ is a coupling of $\marg_1,\marg_2$ is also linear. On the other hand, the Monge formulation is non-linear in $\diffeo$, both in the objective and the constraint. In their fundamental work \cite{benamou2000computational}   Benamou and Brenier developed the following \emph{dynamical} formulation of the Monge problem. Given marginals $\marg_1,\marg_2$ consider all possible flows of probability measures $(\prob(t,\cdot))_{t\in [0,1]}$ connecting $\marg_1$ and $\marg_2$, which evolve according to a continuity equation with a driving vector field $\velb$,
\begin{equation}
\label{eq:cont_eq_intro}
\partial_t\prob(t,x)+\nabla(\prob(t,x)\velb(t,x))=0,\qquad \prob(0,\cdot)=\marg_1,\quad \prob(1,\cdot)=\marg_2,
\end{equation}
(we use $\nabla$ as the symbol for both gradient and divergence when the meaning is clear from context, and when it exists we denote by $\prob(t,x)$ the density of the measure $\prob(t,\cdot)$). Of all flows \eqref{eq:cont_eq_intro} we seek the optimal flow in the sense of
\begin{equation}
\label{eq:BB_intro}
\min_{\prob,\velb}\int_0^1\int_{\domainot} \costt(\velb(t,x))\diff\prob(t, x)\diff t\quad\textnormal{such that $(\prob,\velb)$ satisfy \eqref{eq:cont_eq_intro}}. 
\end{equation}
Benamou and Brenier showed\footnote{In their work \cite{benamou2000computational}  they only considered the case $\costt(x)=\frac{|x|^2}{2}$, but the argument can be extended to other convex costs \cite{MR2006306, buttazzo2009optimization}.} that when $\domainot$ is a compact convex set, $\costt$ is convex, and the marginals $\marg_1,\marg_2$ are regular, the Monge problem \eqref{eq:static_ot__monge_intro} and the dynamical problem \eqref{eq:BB_intro} coincide:
\begin{equation}
\label{eq:monge_dynamic_equiv}
\min_{(\prob,\velb)\textnormal{ satisfying \eqref{eq:cont_eq_intro}}}\int_0^1\int_{\domainot} \costt(\velb(t,x))\diff \prob(t,x)\diff t = \min_{\textnormal{ transport maps $\diffeo$ between $\marg_1$ and $\marg_2$}}\int_{\domainot}\costt(\diffeo(x_1)-x_1)\diff \marg_1(x_1).
\end{equation}
Furthermore, beyond the fact that the values of both optimization problems coincide, the dynamical problem \eqref{eq:BB_intro} can be used to find a solution to the Monge problem \eqref{eq:static_ot__monge_intro}. This is done through the Lagrangian perspective on the Eulerian flows \eqref{eq:cont_eq_intro}. Namely, given a vector field $\velb$ we consider the trajectories of individual particles, 
\begin{equation}
\label{eq:lagrangian_BB_intro}
\partial_t\diffeo(t,x)=\velb(t,\diffeo(t,x)),\qquad \diffeo(0,x)=x,
\end{equation}
which are known to satisfy the relation
\begin{equation}
\label{eq:Euler_lagrangian_intro}
\diffeo(t,\cdot)_{\sharp}\marg_1=\prob_t\quad \text{for all }t\in [0,1]. 
\end{equation}
In particular, any flow \eqref{eq:cont_eq_intro} gives rise to a transport map between $\marg_1$ and $\marg_2$, by solving \eqref{eq:lagrangian_BB_intro} and taking $\diffeo(1,\cdot)$ as the transport map between $\marg_1$ and $\marg_2$. If we now take a vector field $\velb_{\textnormal{opt}}$ which minimizes \eqref{eq:BB_intro}, then one can show that the corresponding $\diffeo_{\textnormal{opt}}(1,\cdot)$ is a solution of  \eqref{eq:static_ot__monge_intro}.  The intuition behind this fact is that given an optimal transport map $\diffeo_{\textnormal{opt}}$ of the Monge problem, the best trajectory for a particle starting at $x$ in the dynamical problem is to flow along straight lines,
\begin{equation}
\label{eq:best_lagrangian_flow_intro}
\diffeo_{\textnormal{opt}}(t,x):=(1-t)x+t\diffeo_{\textnormal{opt}}(x).
\end{equation}
Starting from \eqref{eq:best_lagrangian_flow_intro}, one can reverse engineer the vector field $\velb_{\textnormal{opt}}$ in \eqref{eq:lagrangian_BB_intro}, and then argue that it must be optimal in the dynamical problem \eqref{eq:BB_intro}. 

At this point, we only have another formulation of the Monge problem which, while interesting theoretically, is not necessarily easier to solve numerically.  The insight of \cite{benamou2000computational} was that the dynamical problem \eqref{eq:BB_intro} can be turned into a \emph{convex} optimization problem by a simple change of variables. Given $(\prob,\velb)$ define the momentum $\m:=\prob\velb$, and plug it into \eqref{eq:BB_intro} to get the problem
 \begin{equation}
\label{eq:BB_convex_intro}
\min_{\prob,\m}\int_0^1\int_{\domainot} \costt\left(\frac{\m(t,x)}{\prob(t,x)}\right)\diff\prob(t, x)\diff t\quad\textnormal{such that}\quad\partial_t\prob+\nabla \m=0,\quad \prob(0,\cdot)=\marg_1,\quad \prob(1,\cdot)=\marg_2. 
\end{equation}
When $\costt$ is convex the objective $(\prob,\m)\mapsto \int_0^1\int_{\domainot} \costt\left(\frac{\m(t,x)}{\prob(t,x)}\right)\diff\prob(t, x)\diff t$ is also convex, and further, the constraints are now linear in $(\prob,\m)$, which makes \eqref{eq:BB_convex_intro} a convex optimization problem. Thus, the dynamical formulation of the two-marginal optimal transport problem turned the non-convex Monge problem into a convex problem. This has been exploited already in \cite{benamou2000computational}, and further developed in \cite{papadakis2014optimal},  which used proximal splitting methods to numerically solve the convex problem \eqref{eq:BB_convex_intro}.

\subsection{A dynamical formulation of multi-marginal optimal transport}

Multi-marginal optimal transport is a natural generalization of the two-marginal optimal transport problem \eqref{eq:static_ot_intro}, where we consider $\kk\ge 2$ marginals $\marg_1,\ldots,\marg_{\kk}$ on $\domainot\subseteq\R^{\dd}$. It arises in a broad set of applications from matching theory \cite{carlier2010matching, chiappori2010hedonic}, to Wasserstein barycenters \cite{agueh2011barycenters}, and density functional theory \cite{buttazzo2012optimal, cotar2013density}, as well as other applications \cite[\S 3.3]{pass2015multi}. The  \emph{static} multi-marginal  optimal transport problem with cost $\cost:(\R^{\dd})^{\kk}\to \R$, for marginals $\marg_1,\ldots,\marg_{\kk}$ on $\domainot\subseteq\R^{\dd}$, is 
\begin{equation}
\label{eq:static_mmot_intro}
\min_{\textnormal{ joint couplings $\joint$ of } \marg_1,\ldots,\marg_{\kk}}\int_{\domainot^{\kk}}\cost(x_1,\ldots,x_{\kk})\diff\joint(x_1,\ldots, x_{\kk}).
\end{equation}
The static multi-marginal optimal transport problem also has a Monge formulation\footnote{The choice to distinguish $\marg_1$ as the source measure is just a convention and is not essential.}
\begin{equation}
\label{eq:monge_mmot_intro}
\min_{\textnormal{ transport maps $\diffeo_2,\ldots, \diffeo_{\kk}$ such that $(\diffeo_l)_{\sharp}\marg_1=\marg_l$ for $l=2,\ldots,\kk$}}\int_{\domainot}\cost(x_1,\diffeo_2(x_1)\ldots,\diffeo_{\kk}(x_1))\diff \marg_1( x_1).
\end{equation}
If we try to mimic the dynamical formulation of the two-marginal case in the multi-marginal setting we encounter  a problem. In the former case we need to find a flow between two measures $\marg_1$ and $\marg_2$, so the ``arrow of time" is clear, going from $\marg_1$ to $\marg_2$. In contrast, no such natural  directionality exists for the multi-marginal optimal transport problem.

The main result of this work is that we can in fact find a dynamical formulation of the multi-marginal optimal transport problem, which is equivalent to the static formulation \eqref{eq:static_mmot_intro}. Further, this dynamical problem can be formulated as a  \emph{convex} optimization problem, a fact whose ramifications will be explained below. The key insight behind our formulation is that, already in the two-marginal setting, there is a dynamical formulation, \emph{different from the  Benamou-Brenier formulation \eqref{eq:BB_intro}}, which is equivalent to the static formulation. It is this new  dynamical formulation for the two-marginal case which generalizes to the multi-marginal setting. Moreover, while the Benamou-Brenier formulation  applies only to convex cost functions which are of the form 
\begin{equation}
\label{eq:BB_cost_form}
\cost(x_1,x_2)=\costt(x_2-x_1),\qquad\textnormal{for some convex function}\qquad \costt:\R^{\dd}\to \R,
\end{equation}
our formulation will apply to a broader family  of cost functions.

In contrast to the flow  \eqref{eq:cont_eq_intro}, which is a flow of \emph{marginals}, our new dynamical formulation will be based upon  flows of \emph{couplings}. These flows of couplings will start at a source measure $\base$, and terminate at a coupling of the marginals $\{\marg_l\}_{l=1}^{\kk}$. The role of the source measure is intertwined with properties of the cost function via the following definition. 
\begin{definition}[$\base$-translation-invariant cost functions]
\label{def:diagonal_trans_inv_intro}
Given a probability measure $\base$ on $\domainot^k$ we say that a cost function $\cost:\R^{\kk\dd}\to \R$ is  \emph{$\base$-translation-invariant} if,
\begin{equation}
\label{eq:p_trans_inv_def_intro}
\cost(x-\tvec)=\cost(x)\quad\textnormal{for all} \quad x\in \domainot^{\kk}\quad\textnormal{and $\base$-almost-all}\quad \tvec\in \domainot^{\kk}.
\end{equation}
\end{definition}

To get some intuition for Definition \ref{def:diagonal_trans_inv_intro} let us consider  some examples.
\begin{example}
\label{ex:trivial_cost}
If $\cost$ is identically a constant then $\cost$ will be $\base$-translation-invariant for any $\base$.
\end{example}
\begin{example}
\label{ex:p_delta}
Suppose $0\in\domainot^{\kk}$ and let $\base=\delta_0$. Then any cost $\cost$ is $\base$-translation-invariant.
\end{example}
\begin{example}
\label{ex:trans_inv_cost}
If $\cost$ is translation-invariant, in the sense that
\begin{equation}
\label{eq:trans_inv_def_full_intro}
\cost(x_1-\xi,\ldots,x_{\kk}-\xi)=\cost(x_1,\ldots,x_{\kk})\quad \quad\textnormal{for all}\quad x_1,\ldots,x_k\in \domainot\quad\textnormal{and}\quad \xi\in \R^{\dd},
\end{equation}
then $\cost$ is $\base$-translation-invariant for any $\base$ of the form
\begin{equation}
\label{eq:diagonal_def_intro}
\base=\law(\xi,\ldots,\xi) \quad \textnormal{with}\quad \textnormal{ $\xi\sim \qm$ for some probability measure $\qm$ over $\domainot$}.
\end{equation}
\end{example}

\begin{example}
Let $E\subseteq \R^{\kk\dd}$ be a subspace, and let $\proj_E$ be the orthogonal projection onto $E$, i.e.,  $\proj_E(x):=\argmin_{y\in E} |x-y|^2$. Suppose that $\cost$ is of the form 
\[
\cost(x)=\tilde{\cost}(\proj_E(x))\qquad\textnormal{for a function $\tilde{\cost}:E\to \R$}. 
\]
Then $\cost$ is $\base$-translation-invariant for any $\base$ such that $\supp(\base)\subseteq E^{\perp}$.
\end{example}

We can now state our main result.  

\begin{theorem}[Informal; Theorem \ref{thm:static_dynamic_equiv}]
\label{thm:static_dynamic_equiv_intro}
Let $\domainot\subseteq \R^{\dd}$ be a compact convex set, and let $\marg_1,\ldots,\marg_{\kk}$ be probability measures on $\domainot$. Fix a probability measure $\base$ on $\domainot^{\kk}$, and let  $\cost:\R^{\kk\dd}\to \R$ be a convex function which is $\base$-translation-invariant. Then,
\begin{align}
\label{eq:static_dynamic_equiv_informal}
\begin{split}
&\min_{\textnormal{joint couplings $\joint$ of } \marg_1,\ldots,\marg_{\kk}}\int_{\domainot^{\kk}}\cost(x_1,\ldots,x_{\kk})\diff \joint( x_1,\ldots, x_{\kk})=\min_{\flow,\vel}\int_0^1\int_{\domainot^{\kk}}\cost(\vel(t,x))\diff\flow(t, x)\diff t,
\end{split}
\end{align}
where the minimum on the right-hand side of \eqref{eq:static_dynamic_equiv_informal} is over $(\flow,\vel)$ satisfying
\begin{equation}
\label{eq:cont_eq_coupling_thm_intro}
\partial_t\flow+\nabla(\flow\vel)=0,\quad \flow(0,\cdot)=\base,\quad \textnormal{marginals of $\flow(1,\cdot)$ are $\marg_1,\ldots,\marg_{\kk}$}.
\end{equation}
\end{theorem}
We will elaborate below on the interplay between the source measure $\base$ and the  properties of the cost $\cost$. But first, let us discuss Theorem \ref{thm:static_dynamic_equiv_intro} in general terms. The optimization problem on the right-hand side of \eqref{eq:static_dynamic_equiv_informal} is of a completely different nature compared to the Benamou-Brenier optimization problem \eqref{eq:BB_intro}. In the latter, the source $\marg_1$ and target $\marg_2$ are fixed, and the flow $(\prob(t,\cdot))_{t\in [0,1]}$ interpolates in the space of marginals between $\marg_1$ and $\marg_2$ in an optimal way. In contrast, in \eqref{eq:static_dynamic_equiv_informal} the target $\flow(1,\cdot)$ is not fixed, but only its marginals $\marg_1,\ldots,\marg_{\kk}$ are. The dynamical optimization problem in \eqref{eq:static_dynamic_equiv_informal} is to find a flow $(\flow(t,\cdot))_{t\in [0,1]}$ which starts at time $t=0$ at the source measure $\prob$, whose marginals do not necessarily match $\marg_1,\ldots,\marg_{\kk}$, and gradually deform it so that at time $t=1$ the marginals of $\flow(1,\cdot)$ agree with $\marg_1,\ldots,\marg_{\kk}$. Since the deformation has to be done in an optimal way, in the sense of minimizing the cost of the driving vector field $\vel$, the resulting coupling will be optimal. Indeed, we will show that
\begin{equation}
\begin{split}
\label{eq:optimizer_static_dynamic_intro}
& \flowop\in\argmin_{\flow,\vel}\int_0^1\int_{\domainot^{\kk}}\cost(\vel(t,x))\diff \flow(t,x)\diff t\quad\textnormal{such that $(\flow,\vel)$ satisfy \eqref{eq:cont_eq_coupling_thm_intro}}\\
 &\Longrightarrow \\
 & \flowop(1,\cdot)\in\argmin_{\textnormal{joint couplings $\joint$ of } \marg_1,\ldots,\marg_{\kk}}\int_{\domainot^{\kk}}\cost(x_1,\ldots,x_{\kk})\diff \joint( x_1,\ldots, x_{\kk}).
\end{split}
\end{equation}
In words, solutions to the dynamical problem in \eqref{eq:static_dynamic_equiv_informal} yield solutions to the static  problem in \eqref{eq:static_dynamic_equiv_informal}. (In the proof of Theorem \ref{thm:static_dynamic_equiv} we will see that the arrow in \eqref{eq:optimizer_static_dynamic_intro} can be reversed.)

With Theorem \ref{thm:static_dynamic_equiv_intro} in hand, we can now trade off the generality of the source measures $\base$ and the cost functions $\cost$. 
\begin{corollary}[Informal; Corollary \ref{cor:static_dynamic_equiv}]
\label{cor:static_dynamic_equiv_intro}
Let $\domainot\subseteq\R^{\dd}$ be a compact convex set, and let $\marg_1,\ldots,\marg_{\kk}$ be probability measures on $\domainot$.  Suppose that one of the following holds. 
\begin{enumerate}[(i)]
\item  \label{enum:delta_intro}  $\cost:\R^{\kk\dd}\to \R$ is convex and $\base=\delta_0$ where we assume $0\in\domainot^{\kk}$. 
\item \label{enum:tran_inv_intro}   $\cost:\R^{\kk\dd}\to \R$ is convex, satisfying
\[
\cost(x_1-\xi,\ldots,x_{\kk}-\xi)=\cost(x_1,\ldots,x_{\kk})\quad \quad\textnormal{for all}\quad x_1,\ldots,x_k\in \domainot \quad\textnormal{and}\quad \xi\in \R^{\dd},
\]
and  $\base$  is any probability measure on $\domainot^{\kk}$ of the form 
\[
\base=\law(\xi,\ldots,\xi),\qquad \xi\sim \qm\quad\textnormal{for some probability measure $\qm$ on $\domainot$}.
\]
\item \label{enum:proj_intro}
 $\cost:\R^{\kk\dd}\to \R$ is of the form
\[
\cost(x)=\tilde{\cost}(\proj_E(x))\qquad\textnormal{for a convex function $\tilde{\cost}:E\to \R$}, 
\]
and $\base$ is a measure on $\domainot$ such that $\supp(\base)\subseteq E^{\perp}$.
\end{enumerate}
In each of the cases \eqref{enum:delta_intro}-\eqref{enum:tran_inv_intro}-\eqref{enum:proj_intro}  we have
\begin{equation}
\label{eq:static_dynamic_equiv_trans_inv_intro}
\begin{split}
&\min_{\textnormal{joint couplings $\joint$ of } \marg_1,\ldots,\marg_{\kk}}\int_{\domainot^{\kk}}\cost(x_1,\ldots,x_{\kk})\diff \joint( x_1,\ldots, x_{\kk})=\min_{\flow,\vel}\int_0^1\int_{\domainot^{\kk}}\cost(\vel(t,x))\diff\flow(t, x)\diff t,
\end{split}
\end{equation}
where the minimum on the right-hand side of \eqref{eq:static_dynamic_equiv_trans_inv_intro} is over $(\flow,\vel)$ satisfying
\begin{equation}
\label{eq:cont_eq_coupling_intro}
\partial_t\flow+\nabla(\flow\vel)=0,\quad \flow(0,\cdot)=\base,\quad \textnormal{marginals of $\flow(1,\cdot)$ are $\marg_1,\ldots,\marg_{\kk}$}.
\end{equation}
\end{corollary}

\begin{remark}[Semi-convex cost functions; Corollary \ref{cor:static_dynamic_equiv_semicvx}]
\label{rem:semicvx_intro}
When $\base=\delta_0$ we can in fact relax the convexity assumption on the cost. In particular, if $\cost$ is $\scvx$-semi-convex, i.e., 
\begin{equation}
\label{eq:def_semicvx_intro}
\textnormal{there exists a constat $\scvx\ge 0$ such that $\R^{\kk\dd}\ni x\mapsto \cost(x)+\scvx\frac{|x|^2}{2}$ is convex},
\end{equation}
then Corollary \ref{cor:static_dynamic_equiv_intro} implies that 
\begin{align}
\label{eq:static_dynamic_equiv_semicvx_informal}
\begin{split}
&\min_{\textnormal{joint couplings $\joint$ of } \marg_1,\ldots,\marg_{\kk}}\int_{\domainot^{\kk}}\cost(x_1,\ldots,x_{\kk})\diff \joint( x_1,\ldots, x_{\kk})\\
&=\min_{\flow,\vel}\int_0^1\int_{\domainot^{\kk}}\left[\cost(\vel(t,x))+\scvx\frac{|\vel(t,x)|^2}{2} \right]\diff\flow(t, x)\diff t-\frac{\scvx}{2}\sum_{l=1}^{\kk}\int_{\domainot} |x_l|^2\diff \marg_l(x).
\end{split}
\end{align}
In other words, when $\base=\delta_0$, the static multi-marginal optimal transport problem has a dynamical formulation, even when the cost $\cost$ is only semi-convex; note that the term $\frac{\scvx}{2}\sum_{l=1}^{\kk}\int_{\domainot} |x_l|^2\diff \marg_l(x)$ depends only on the marginals $\marg_1,\ldots,\marg_{\kk}$. Moreover, 
\begin{equation}
\begin{split}
\label{eq:optimizer_static_dynamic_semiconvx_intro}
& \flowop\in\argmin_{(\flow,\vel)  \textnormal{ satisfy } \eqref{eq:cont_eq_coupling_intro}}\int_0^1\int_{\domainot^{\kk}}\left[\cost(\vel(t,x))+\scvx\frac{|\vel(t,x)|^2}{2} \right]\diff\flow(t, x)\diff t-\frac{\scvx}{2}\sum_{l=1}^{\kk}\int_{\domainot} |x_l|^2\diff \marg_l(x)\\
 &\Longrightarrow \\
 & \flowop(1,\cdot)\in\argmin_{\textnormal{joint couplings $\joint$ of } \marg_1,\ldots,\marg_{\kk}}\int_{\domainot^{\kk}}\cost(x_1,\ldots,x_{\kk})\diff \joint( x_1,\ldots, x_{\kk}).
\end{split}
\end{equation}
\end{remark}

The power of Corollary \ref{cor:static_dynamic_equiv_intro}\eqref{enum:delta_intro} is that if we restrict to a point-mass source measure, we can handle all (semi-)convex cost functions. This result is new even in the two-marginal setting, since the Benamou-Brenier formulation \eqref{eq:monge_dynamic_equiv} requires convex cost functions which are translation-invariant as in \eqref{eq:trans_inv_def_full_intro}. On the other hand, if we are willing to make the  assumptions in Corollary \ref{cor:static_dynamic_equiv_intro}\eqref{enum:tran_inv_intro}, which hold for the common costs of the form
\[
\cost(x_1,\ldots,x_{\kk})=\sum_{i,j=1}^{\kk}\costt(x_i-x_j), 
\]
for a cost function $\costt:\domainot\to \R$, then we can consider a much larger class of source measures. The setting of Corollary \ref{cor:static_dynamic_equiv_intro}\eqref{enum:proj_intro} can be thought of as a middle ground between \eqref{enum:delta_intro} and  \eqref{enum:tran_inv_intro}. This  flexibility in choosing the source measure is beneficial when performing numerics. 

Let us conclude this section with an explicit  example demonstrating the difference between the Benamou-Brenier dynamical problem and our dynamical problem, already in the two-marginal translation-invariant cost setting.

\begin{example}[Quadratic cost]
\label{ex:quadratic}
One of the most classical costs for the two-marginal optimal transport problem is the quadratic cost,
\begin{equation}
\label{eq:quadratic_2marg_intro}
\costt(x_1-x_2)=\frac{|x_1-x_2|^2}{2},
\end{equation}
where $|\cdot|$ is the Euclidean norm in $\R^{\dd}$. The Benamou-Brenier dynamical problem \eqref{eq:BB_intro} then reads

\begin{equation}
\label{eq:BB_quadratic_intro}
\min_{\prob,\velb}\int_0^1\int_{\domainot} \frac{|\velb(t,x)|^2}{2}\diff \prob(t,x)\diff t,\qquad \partial_t\prob+\nabla(\prob\velb)=0,\qquad \prob(0,\cdot)=\marg_1,\quad \prob(1,\cdot)=\marg_2.
\end{equation}
On the other hand, our dynamical problem in \eqref{eq:static_dynamic_equiv_informal} reads, with $\cost(x_1,x_2)=\costt(x_1-x_2)=\frac{|x_1-x_2|^2}{2}$, and $\vel=(\vel_1,\vel_2)$,
\begin{equation}
\label{eq:mmot_2marg_quadratic_intro}
\min_{\flow,\vel}\int_0^1\int_{\domainot^2} \frac{|\vel_1(t,x)-\vel_2(t,x)|^2}{2}\diff\flow(t, x)\diff t,
\end{equation}
over $(\flow,\vel)$ satisfying
\begin{equation}
\label{eq:mmot_2marg_quadratic_constraint_intro}
\partial_t\flow+\nabla(\flow\vel)=0,\qquad \flow(0,\cdot)=\base,\qquad \textnormal{marginals of $\flow(1,\cdot)$ are $\marg_1,\marg_2$},
\end{equation}
where $\base$ is any measure as in \eqref{eq:diagonal_def_intro}. 

The quadratic cost generalizes to the multi-marginal setting with the cost $\cost:\R^{\kk\dd}\to [0,\infty)$,
\begin{equation}
\label{eq:quadratic_cost_mmot_intro}
\cost(x_1,\ldots,x_{\kk})=\sum_{i,j=1}^{\kk}\frac{|x_i-x_j|^2}{2},
\end{equation}
and the dynamical formulation \eqref{eq:mmot_2marg_quadratic_intro} then generalizes mutatis mutandis as, with $\vel=(\vel_1,\ldots, \vel_{\kk})$, 
\begin{equation}
\label{eq:mmot_quadratic_intro}
\min_{\flow,\vel}\sum_{i,j=1}^{\kk}\int_0^1\int_{\domainot^{\kk}} \frac{|\vel_i(t,x)-\vel_j(t,x)|^2}{2}\diff\flow(t, x)\diff t,
\end{equation}
over $(\flow,\vel)$ satisfying
\begin{equation}
\label{eq:mmot_quadratic_constraint_intro}
\partial_t\flow+\nabla(\flow\vel)=0,\qquad \flow(0,\cdot)=\base,\qquad \textnormal{marginals of $\flow(1,\cdot)$ are $\marg_1,\ldots,\marg_{\kk}$},
\end{equation}
where $\base$ is any measure as in \eqref{eq:diagonal_def_intro}.
\end{example}

\subsection{Convex formulations and quasi-Monge solutions}
As explained in Section \ref{subsec:two_marg}, a numerical advantage of the Benamou-Brenier dynamical formulation is that after a change of variables it becomes a convex optimization problem \eqref{eq:BB_convex_intro}, which in turns yields solutions to the Monge problem \eqref{eq:static_ot__monge_intro}; see Equations \eqref{eq:lagrangian_BB_intro}-\eqref{eq:Euler_lagrangian_intro} and the subsequent discussion. It turns out that the same change of variables can be applied to our dynamical formulation of multi-marginal optimal transport  to yield a convex optimization problem. Indeed, setting $\m:=\flow\vel$ we see that the dynamical formulation of Theorem \ref{thm:static_dynamic_equiv_intro} becomes
\begin{equation}
\label{coupling_flow_convex_intro}
\begin{split}
&\min_{\flow,\m}\int_0^1\int_{\domainot^{\kk}}\cost\left(\frac{\diff \m(t,x)}{\diff\flow(t, x)}\right)\diff\flow(t, x)\diff t\\
&\textnormal{such that}\quad \partial_t\flow+\nabla\m=0,\quad \flow(0,\cdot)=\base,\quad \textnormal{marginals of $\flow(1,\cdot)$ are $\marg_1,\ldots,\marg_{\kk}$}.
\end{split}
\end{equation}
The convexity of $\cost$ ensures that the objective $(\flow,\m)\mapsto \int_0^1\int_{\domainot^{\kk}}\cost\left(\frac{\diff\m(t, x)}{\diff\flow(t, x)}\right)\diff \flow(t,x)\diff t$ is convex, and the constraints on $(\flow,\m)$ are linear, so  \eqref{coupling_flow_convex_intro} is a convex optimization problem. 

Unlike the two-marginal setting where the existence of Monge solutions is guaranteed under fairly broad conditions, the existence of Monge solutions in multi-marginal optimal transport is much less understood. Specifically, under the assumptions on the cost function in Theorem \ref{thm:static_dynamic_equiv_intro}, it is an open problem whether Monge solutions even exist. On the other hand, \emph{quasi-Monge} solutions always exist. By a ``quasi-Monge solution" to the  static multi-marginal optimal transport problem \eqref{eq:static_mmot_intro} we mean that there exists a probability measure $\qm$ on $\domainot\subseteq\R^{\dd}$, and transport maps $\diffeo_1,\ldots,\diffeo_{\kk}:\domainot\to\domainot$, such that
\begin{equation}
\label{eq:quasiMonge_def_intro}
(\diffeo_l)_{\sharp}\qm = \marg_l\quad\textnormal{for all}\quad l=1,\ldots,\kk,
\end{equation}
and 
\begin{equation}
\label{eq:quasiMonge_opt_intro}
\int_{\domainot}\cost(\diffeo_1(\xi),\ldots,\diffeo_{\kk}(\xi))\diff\qm( \xi) =\min_{\textnormal{ joint couplings $\joint$ of } \marg_1,\ldots,\marg_{\kk}}\int_{\domainot^{\kk}}\cost(x_1,\ldots,x_{\kk})\diff  \joint(x_1,\ldots, x_{\kk}).
\end{equation}
Note that a quasi-Monge solution will be a Monge solution if $\qm=\marg_1$ and $\diffeo_1$ is the identity mapping. The concept of quasi-Monge solutions was developed by Friesecke and V\"{o}gler \cite{friesecke2018breaking} (see also \cite[\S 4.3.14]{friesecke2022strong}) in the context of \emph{discrete} multi-marginal optimal transport problems, where the marginals are all identical, $\marg_1=\cdots=\marg_{\kk}$, a setting which arises in density functional theory with the non-convex repulsive Coulomb cost\footnote{Due to the assumption that the marginals are all identical, and the cost is symmetric with respect to permutations of the variables, the definition of  quasi-Monge solutions in \cite{friesecke2018breaking} is slightly different, but the essential idea is the same.}. We use the same terminology for the continuous setting and extend it to settings where the marginals $\marg_1,\ldots,\marg_{\kk}$ are different. While in the discrete setting of \cite{friesecke2018breaking} the existence of optimal quasi-Monge solutions is non-trivial, in our continuous setting the existence of optimal quasi-Monge solutions is more straightforward. Indeed, given a non-atomic  measure $\qm$ on $\domainot$, and a minimizer $\jointop$ of the right-hand side of \eqref{eq:quasiMonge_opt_intro}, it follows from \cite[Theorem A]{MR2288266} that there exists a transport map $\diffeo=(\diffeo_1,\ldots,\diffeo_{\kk}):\domainot\to \domainot^{\kk}$ such that 
\begin{equation}
\label{eq:quasi_monge_exist_intro}
\law(\diffeo_1(\xi),\ldots,\diffeo_{\kk}(\xi))=\jointop\qquad\textnormal{where}\qquad \xi\sim \qm. 
\end{equation}
By construction, \eqref{eq:quasi_monge_exist_intro} gives an optimal quasi-Monge solution to the static multi-marginal optimal transport problem. 

Let us now explain heuristically the relation between solutions to our dynamical multi-marginal optimal transport problem and quasi-Monge solutions, in the setting of Corollary \ref{cor:static_dynamic_equiv_intro}\eqref{enum:tran_inv_intro}. Given an optimal vector field  $\velop$ of the dynamical problem, solve the differential equation
\begin{equation}
\label{eq:qm_lagrange_intro}
\partial_t\diffeo_{\textnormal{opt}}(t,x)=\velop(\diffeo_{\textnormal{opt}}(t,x)),\qquad \diffeo_{\textnormal{opt}}(0,x)=x. 
\end{equation}
As the source measure $\base$ is of the form
\begin{equation}
\label{eq:qm_diagona_intro}
p=\law(\xi,\ldots,\xi) \qquad\textnormal{where}\qquad \xi\sim \qm, \quad\textnormal{with $\qm$ non-atomic probability measure on $\domainot$},
\end{equation}
we find that
\begin{equation}
\label{eq:qm_lagrange_t1_intro}
\diffeo_{\textnormal{opt}}(1,\cdot)_{\sharp}\base=\flowop(1,\cdot),
\end{equation}
where we recall  that $\flowop(1,\cdot)$ is a minimizer of the  static multi-marginal  optimal transport problem; see Equation \eqref{eq:optimizer_static_dynamic_intro}. Thus, $\flowop(1,\cdot)$ is precisely of the form \eqref{eq:quasi_monge_exist_intro}, i.e., solutions to the \emph{dynamical} multi-marginal  optimal transport problem are quasi-Monge solutions to  the   \emph{static} multi-marginal optimal transport problem. This correspondence mirrors in the multi-marginal setting the relation between solutions to the two-marginal  Benamou-Brenier dynamical problem and Monge solutions to the static two-marginal  optimal transport problem. Since Monge solutions may not exist in the multi-marginal setting, we need to resort to quasi-Monge solutions, which are exactly captured by our  dynamical  multi-marginal optimal transport problem. Further, just as in the two-marginal case the Benamou-Brenier formulation yields a convex optimization problem which solves the non-convex Monge problem, \eqref{coupling_flow_convex_intro} gives  a convex formulation of our dynamical  multi-marginal optimal transport problem which solves the non-convex quasi-Monge problem (for translation-invariant cost functions as in  \eqref{eq:trans_inv_def_full_intro}).

While \eqref{eq:qm_lagrange_intro} sheds light on the relation between our dynamical formulation and quasi-Monge solutions, we cannot in fact use the relation \eqref{eq:qm_lagrange_t1_intro} to construct quasi-Monge solutions. The reason is that,  generically, the regularity of $\velop$ will be low, since we use  flows of couplings which can be supported on low-dimensional sets, so \eqref{eq:qm_lagrange_intro} will not have a classical solution. On the theoretical side, we will employ  some technical tools to overcome this challenge. On the numerical side, rather than attempting to solve  \eqref{eq:qm_lagrange_intro} to get an estimate of the transport map $\diffeo_{\textnormal{opt}}$, it is easier to simply estimate the transport map from $\flowop(1,\cdot)$ by taking conditional expectations; see Section \ref{subsec:numerics}. 

To conclude this section, in Figure \ref{fig:3marginals} and  Figure \ref{fig:transport_maps} we consider the example of $\kk=3$ marginals in dimension $\dd=1$ with the quadratic cost \eqref{eq:quadratic_cost_mmot_intro}. In this setting there are in fact Monge solutions, which can be computed analytically, and we compare them to the transport maps obtained from $\flowop(1,\cdot)$, which we compute using a proximal splitting  method; see Section \ref{subsec:numerics} for details. 

\begin{figure}[htbp]
  \centering

  \begin{subfigure}[t]{0.31\textwidth}
    \centering
    \includegraphics[width=\linewidth]{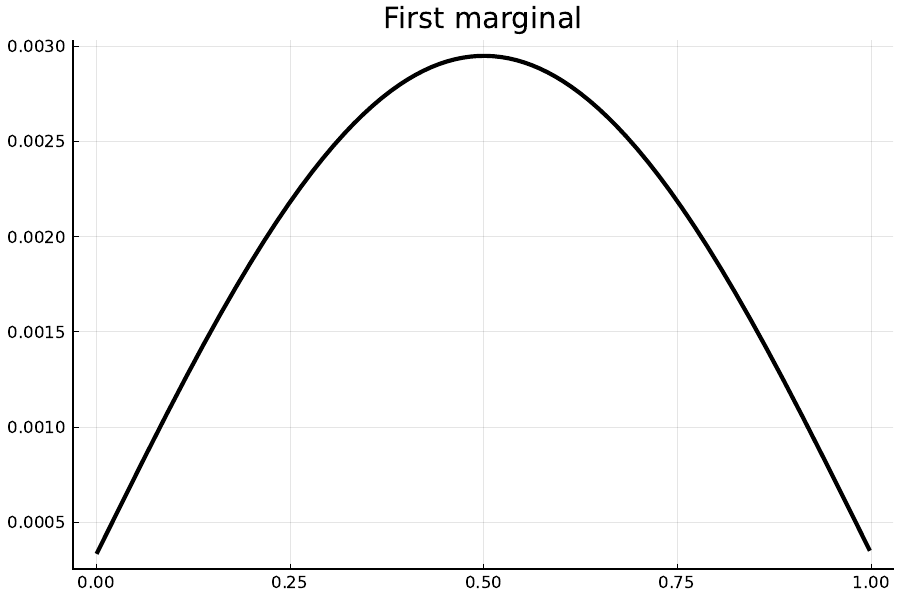}
    \caption{Density of $\marg_1$}
    \label{fig:mu1}
  \end{subfigure}
  \begin{subfigure}[t]{0.31\textwidth}
    \centering
    \includegraphics[width=\linewidth]{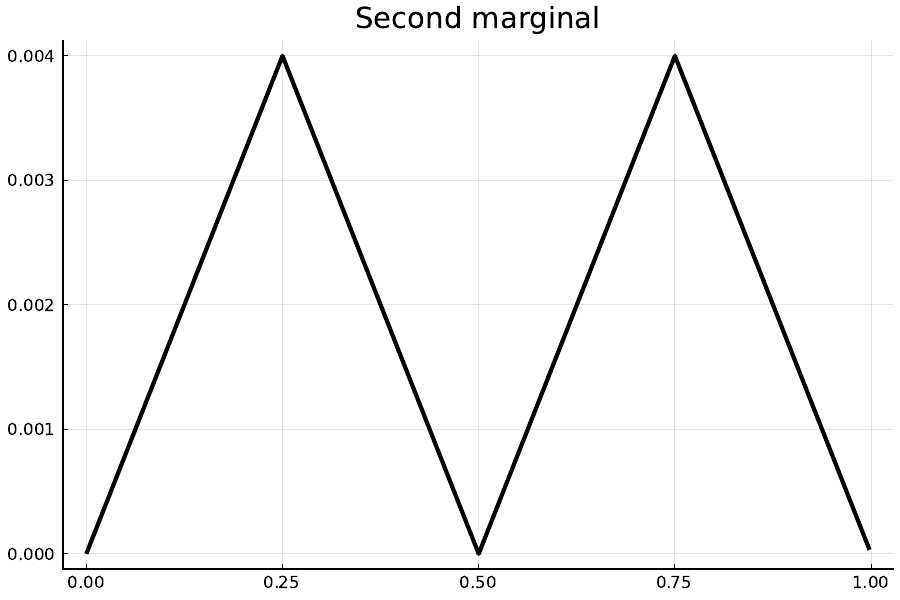}
    \caption{Density of $\marg_2$}
    \label{fig:mu2}
  \end{subfigure}
  \begin{subfigure}[t]{0.31\textwidth}
    \centering
    \includegraphics[width=\linewidth]{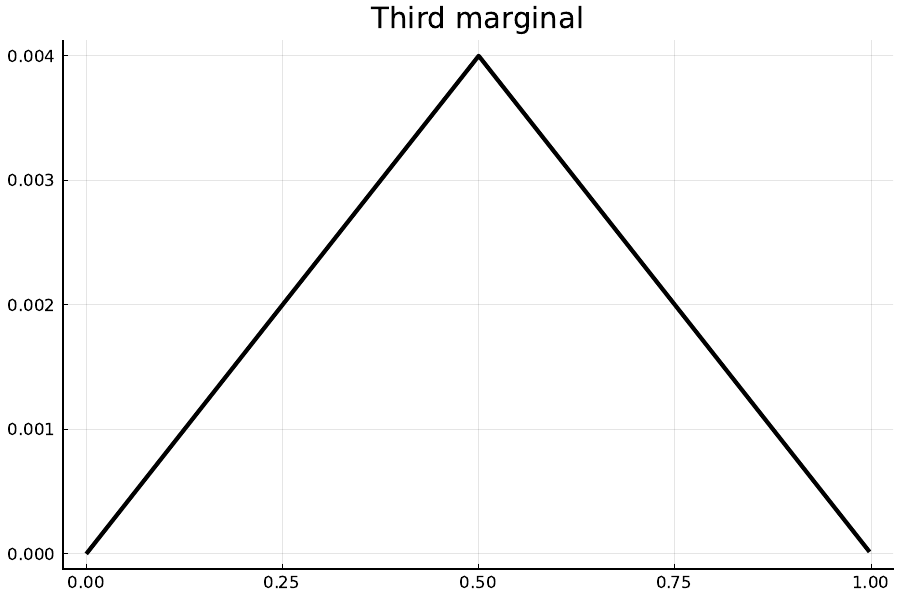}
    \caption{Density of $\marg_3$}
    \label{fig:mu3}
  \end{subfigure}

  \caption{In our numerical experiment we solve the   dynamical multi-marginal optimal transport with quadratic cost \eqref{eq:quadratic_cost_mmot_intro} for the 3 marginals $\marg_1,\marg_2,\marg_3$.}
  \label{fig:3marginals}
\end{figure}

\begin{figure}[htbp]
  \centering

  \begin{subfigure}[t]{0.48\textwidth}
    \centering
    \includegraphics[width=\linewidth]{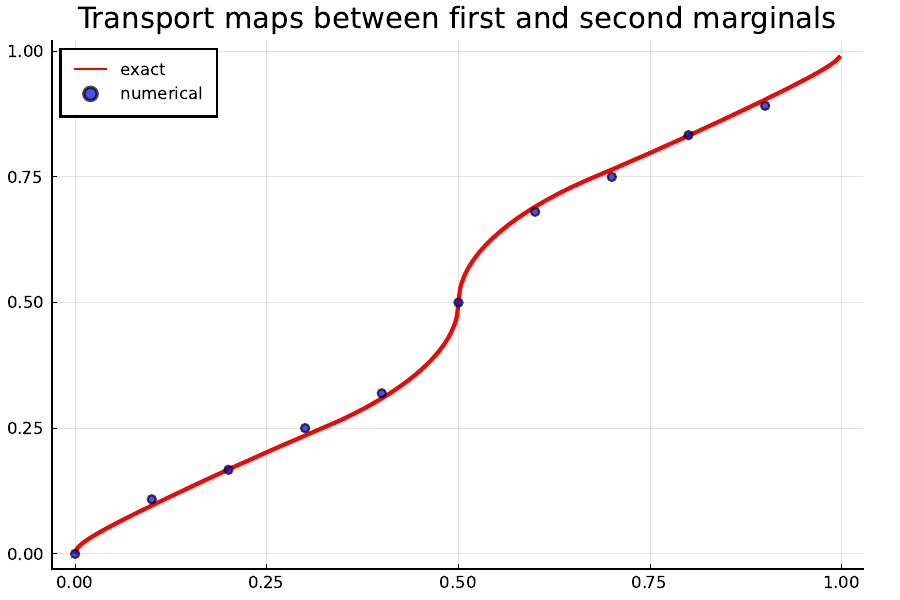}
    \caption{Analytic and numerical transport maps between $\marg_1$ and $\marg_2$.}
    \label{fig:transport12}
  \end{subfigure}
  \hfill
  \begin{subfigure}[t]{0.48\textwidth}
    \centering
    \includegraphics[width=\linewidth]{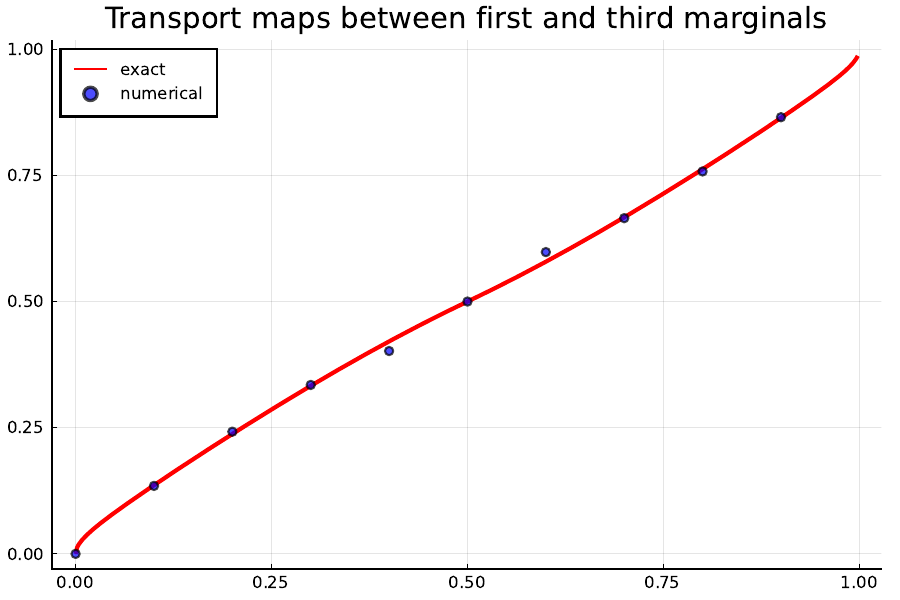}
    \caption{Analytic and numerical transport maps between $\marg_1$ and $\marg_3$.}
    \label{fig:transport13}
  \end{subfigure}

  \caption{The optimal coupling of the  static multi-marginal optimal transport with quadratic cost for the marginals $\marg_1,\marg_2,\marg_3$  of Figure \ref{fig:3marginals} can be computed analytically in terms of transport maps between $\marg_1$ to $\marg_2$ (res. $\marg_3$), which are expressed in terms of cumulative distribution functions; see Equation \eqref{eq:analytic_transport_GS}. In the figure we compare these exact solutions to the solutions obtained from our algorithm; see Equation \eqref{eq:transport_approx}.}
  \label{fig:transport_maps}
\end{figure}

\subsection{Duality}
From the beginning of the Kantorovich theory of optimal transport, duality played a crucial role. For the static two-marginal  optimal transport problem, duality takes the form
\begin{equation}
\label{eq:duality_static_2marg_intro}
\begin{split}
&\min_{\textnormal{ joint couplings $\joint$ of } \marg_1,\marg_2}\int_{\domainot^2}\costt(x_2-x_1)\diff \joint( x_1,x_2)\\
&=\\
&\max_{\{\lm_1,\lm_2:\domainot\to \R\,|\,\lm_1(x_1)+\lm_2(x_2)\le \costt(x_2-x_1)\}}\left\{\int_{\domainot}\lm_1(x_1)\diff\marg_1( x_1)+\int_{\domainot}\lm_2(x_2)\diff \marg_2(x_2)\right\},
\end{split}
\end{equation}
while for the Benamou-Brenier dynamical  formulation it takes the form,
\begin{equation}
\label{eq:duality_dynamic_2marg_intro}
\begin{split}
&\min_{\{(\prob,\velb)\,|\, \partial_t\prob+\nabla(\prob\velb)=0,~\prob(0,\cdot)=\marg_1,~\prob(1,\cdot)=\marg_2\}}\int_0^1\int_{\domainot} \costt(\velb(t,x))\diff\prob(t, x)\diff t\\
&=\\
&\max_{\{\lm:[0,1]\times\domainot\to \R\,|\,\partial_t\lm+\costt^*(\nabla\lm)\le 0\}}\left\{\int_{\domainot}\lm(1,x)\diff\marg_1( x)-\int_{\domainot}\lm(0,x)\diff\marg_2( x)\right\},
\end{split}
\end{equation}
where $\costt^*$ is the Legendre-Fenchel transform of $\costt$. Turning to the multi-marginal setting, for the static problem we have the duality
\begin{equation}
\label{eq:duality_static_multimarg_intro}
\begin{split}
&\min_{\textnormal{ joint couplings $\joint$ of } \marg_1,\ldots,\marg_{\kk}}\int_{\domainot^{\kk}}\cost(x_1,\ldots,x_{\kk})\diff  \joint(x_1,\ldots, x_{\kk})\\
&=\\
&\max_{\{\lm_1,\ldots,\lm_{\kk}:\domainot\to \R\, |\, \sum_{l=1}^{\kk}\lm_l(x_l)\le \cost(x_1,\ldots,x_{\kk})\}}\sum_{l=1}^{\kk}\int_{\domainot}\lm_l(x_l)\diff \marg_l(x_l).
\end{split}
\end{equation}
The next result gives the analogue of \eqref{eq:duality_dynamic_2marg_intro} in the multi-marginal setting\footnote{For technical reasons we make additional regularity and integrability assumptions  on the cost functions in our duality result .}. We note that the assumptions in following theorem are satisfied by all the different cases in Corollary \ref{cor:static_dynamic_equiv_intro}. 

\begin{theorem}[Informal; Theorem \ref{thm:dual_dynamical}]
\label{thm:duality_intro}
Let $\domainot\subseteq\R^{\dd}$ be a compact convex set, and let $\base$ be  a probability measure on $\domainot^{\kk}$. Let $\cost:\R^{\kk\dd}\to \R$ be a continuously differentiable cost function satisfying $\lim_{\beta\uparrow+\infty}\frac{1}{\beta}\int_{\R^{\kk\dd}}\cost(\beta x)\diff \base(x)=0$. If $\cost$ is convex and $\base$-translation-invariant, 
then, the dual of the dynamical multi-marginal optimal transport problem is 
\begin{equation}
\begin{split}
\label{eq:dual_thm_intro}
&\sup_{\{\lm:[0,1]\times\domainot^{\kk}\to \R,~\lm_1,\ldots,\lm_{\kk}:\domainot\to \R\}}\left\{\sum_{l=1}^{\kk}\int_{\domainot}\lm_l(x_l)\diff \marg_l(x_l) - \int_{\domainot^{\kk}}\lm(0,x)\diff\base( x)\right\},
\end{split}
\end{equation}
where
\begin{equation}
\label{eq:dual_constraintHJ_intro}
\partial_t\lm(t,x)+\cost^*(\nabla\lm(t,x))\le 0\qquad\textnormal{for all}\qquad t\in [0,1]\quad\textnormal{and}\quad x=(x_1,\ldots,x_{\kk})\in \domainot^{\kk},
\end{equation}
and 
\begin{equation}
\label{eq:dual_constraint_domination_intro}
 \sum_{l=1}^{\kk}\lm_l(x_l)\le \lm(1,x) \qquad\textnormal{for}\quad \textnormal{$(\otimes_{l=1}^{\kk}\marg_l)$-almost all}\qquad (x_1,\ldots,x_{\kk})\in \domainot^{\kk}.
\end{equation}
\end{theorem}
The fact that our dynamical formulation is on the level of flows of couplings, as opposed to flows of marginals, manifests itself also in the form of the dual problem in Theorem \ref{thm:duality_intro}.  In particular, while both \eqref{eq:duality_dynamic_2marg_intro} and  \eqref{eq:dual_constraintHJ_intro} impose the Hamilton-Jacobi equation constraint $\partial_t\lm+\cost^*(\nabla\lm)\le 0$, the term $\lm(1,x)$ in \eqref{eq:duality_dynamic_2marg_intro} is replaced in \eqref{eq:dual_thm_intro}-\eqref{eq:dual_constraint_domination_intro} by $\sum_{l=1}^{\kk}\lm_l(x_l)$, since we only have marginal constraints on $\flow$ at $t=1$. We note that Theorem \ref{thm:duality_intro} does not address the question of attainment of the supremum, which is discussed in Remark   \ref{rem:sup_attained}.

\subsection*{Organization of paper} In Section \ref{sec:DMMOT} we prove our main result on the equivalence of the dynamical and static multi-marginal optimal transport problems (Theorem \ref{thm:static_dynamic_equiv_intro}/Theorem \ref{thm:static_dynamic_equiv}, Corollary \ref{cor:static_dynamic_equiv_intro}/Corollary \ref{cor:static_dynamic_equiv}, and Remark \ref{rem:semicvx_intro}/Corollary \ref{cor:static_dynamic_equiv_semicvx}). In Section \ref{sec:dual} we derive the dual of the dynamical multi-marginal optimal transport problem (Theorem \ref{thm:duality_intro}/Theorem \ref{thm:dual_dynamical}). In Section \ref{sec:proximal} we show how to numerically solve the primal dynamical multi-marginal optimal transport problem using proximal-splitting methods.

\subsection*{Acknowledgments} We are very grateful to Hugo Lavenant for helpful comments on earlier versions of this manuscript, as well as help  with the numerical aspects of proximal splitting methods. We are also very grateful to Gabriel Peyr\'e for his helpful remarks on  proximal splitting methods. We would also like to thank Sinho Chewi, Jonathan Niles-Weed, and the participants of the program Physics of AI Algorithms 2025 at  the Les Houches School of Physics for insightful conversations on this work. 

This material is based upon work supported by the National Science Foundation under Award Numbers DMS-2331920 and DMS- 2508545. Pass is pleased to acknowledge the support of Natural Sciences and
Engineering Research Council of Canada Discovery Grant number 04864-2024.

\section{A dynamical formulation of multi-marginal optimal transport}
\label{sec:DMMOT}

\subsection{Main result}
\label{subsec:main_result}

In this section we state our main result, Theorem \ref{thm:static_dynamic_equiv}, on the equivalence between the static and dynamical multi-marginal optimal transport problems. To formulate the theorem we begin by defining the two problems. Our cost function is  a measurable function $\cost:\R^{\kk\dd}\to \R$, where throughout  this work we make the identification  $(\R^{\dd})^{\kk} \cong\R^{\kk\dd}$, and write $x=(x_1,\ldots,x_{\kk})\in \R^{\kk\dd}$ where $x_l\in\R^{\dd}$ for $l\in [\kk]$ with $[\kk]:=\{1,\ldots,\kk\}$. 

Let us begin with the classical static multi-marginal optimal transport problem.

\begin{definition}[Static multi-marginal optimal transport]
\label{def:static_main}
Let $\marg_1,\ldots,\marg_{\kk}$ be probability measures on a Borel set  $\domainot\subseteq \R^{\dd}$, and let $\Pi(\marg_1,\ldots,\marg_{\kk})$ be the set of probability measures on $\domainot^{\kk}$ whose respective marginals are $\marg_1,\ldots,\marg_{\kk}$.  Denote the static cost of a probability measure $\joint$ on $\domainot^{\kk}$ by
\begin{equation}
\label{eq:acts_def}
\acts(\joint):=\int_{\domainot^{\kk}}\cost(x_1,\ldots,x_{\kk})\diff  \joint(x_1,\ldots,  x_{\kk}).
\end{equation}
The \emph{static multi-marginal  optimal transport problem} is 
\begin{equation}
\label{eq:static_mmot_def}
\inf_{\joint \in \Pi(\marg_1,\ldots,\marg_{\kk})}\acts(\joint). 
\end{equation}
\end{definition}
Next we define our dynamical optimal transport problem for the multi-marginal case.

\begin{definition}[Dynamical multi-marginal  optimal transport]
\label{def:dynamic_main}
Let $\domainot\subseteq \R^{\dd}$ be a Borel set. Let $\mathcal P(\domainot^k)$ be the set of probability measures on $\domainot^k$ endowed with the weak convergence topology, $C\left([0,1]; \mathcal P(\domainot^{\kk})\right)$ the set of continuous maps from $[0,1]$ to $\mathcal P(\domainot^{\kk})$, and let  $\mathcal M([0,1]\times\domainot^{\kk}; \R^{\kk\dd})$ be the set of $\R^{\kk\dd}$-vector-valued measures on $[0,1]\times\domainot^{\kk}$ endowed with the total variation norm. Let $\marg_1,\ldots,\marg_{\kk}$ be probability measures on $\domainot$, and let $\base$ be a  probability measure on $\domainot^{\kk}$. Define the set of flows 
\begin{equation}
\label{eq:cont_set}
\CC(\base,\marg_1,\ldots,\marg_{\kk}):=
\begin{cases}
(\flow ,\m)\in C\left([0,1]; \mathcal P(\domainot^{\kk})\right)\times \mathcal M([0,1]\times\domainot^{\kk};\R^{\kk\dd}); \m\ll\flow,\\
\int_0^1\int_{\domainot^{\kk}}\left|\frac{\diff\m(t, x)}{\diff\flow(t, x)}\right|^p\diff\flow(t, x)\diff t<+\infty\quad\textnormal{for some $p>1$,}\\
\partial_t\flow+\nabla \m=0\textnormal{ in a weak sense},\\
\flow(0, \cdot)=\base\,\textnormal{ and } \,\flow_l(1,\cdot)=\marg_l\textnormal{ for all }l\in [\kk],
\end{cases}
\end{equation}
where $\m\ll\flow$ means that $\m(t,\cdot)$ is absolutely continuous with respect to $\flow(t,\cdot)$ for all $t\in [0,1]$, $\nabla \m$ stands for the divergence of $\m$, $\partial_t\flow+\nabla \m=0$ in a weak sense means that for every continuously differentiable function $\test\in C^1([0,1]\times\domainot^{\kk}; \R)$,
\begin{equation}
\label{eq:weak_cteq}
\int_0^1\int_{\domainot^{\kk}}\partial_t\test(t,x)\diff\flow(t, x)\diff t+\int_0^1\int_{\domainot^{\kk}}\langle\nabla\test(t,x),\diff\m(t, x)\rangle \diff t=\int_{\domainot^{\kk}}\eta(1,x)\diff\flow(1, x)-\int_{\domainot^{\kk}}\eta(0,x)\diff\flow(0, x),
\end{equation}
(note that \eqref{eq:weak_cteq} implicitly imposes Neumann boundary conditions on $\m$), and $\pi_l(t,\cdot)$ stands for the $l$th marginal of $\pi(t,\cdot)$ for $l\in [\kk]$ and $t\in [0,1]$. 

Denote the dynamical cost  of  a flow $(\flow,\m)\in \CC(\base,\marg_1,\ldots,\marg_{\kk})$ by
\begin{equation}
\label{eq:actd_def}
\actd(\flow,\m):=\int_0^1\int_{\domainot^{\kk}}\cost\left(\frac{\diff \m(t,x)}{\diff\flow(t, x)}\right)\diff\flow(t, x)\diff t.
\end{equation}
The \emph{dynamical multi-marginal optimal transport problem} is 
\begin{equation}
\label{eq:dynamical_mmot_def}
\inf_{(\flow,\m)\in \CC(\base,\marg_1,\ldots,\marg_{\kk})}\actd(\flow,\m).
\end{equation}
\end{definition}
Our main result is that under convexity assumptions the static and dynamical multi-marginal optimal transport problems of Definition \ref{def:static_main} and Definition \ref{def:dynamic_main} are equivalent. 
\begin{theorem}
\label{thm:static_dynamic_equiv}
Let $\domainot\subseteq\R^{\dd}$ be a compact convex set, and let $\base$ be  a probability measure on $\domainot^{\kk}$. Let $\cost:\R^{\kk\dd}\to \R$ be a convex function which is  $\base$-translation-invariant, 
\begin{equation}
\label{eq:trans_inv_cost_thm}
\cost(x-\tvec)=\cost(x)\quad\textnormal{for all} \quad x\in \domainot^{\kk}\quad\textnormal{and $\base$-almost-all}\quad \tvec\in \domainot^{\kk}.
\end{equation}
Then,
\begin{equation}
\label{eq:static_dynamic_equiv}
\min_{\joint\in \Pi(\marg_1,\ldots,\marg_{\kk})}\acts(\joint)=\min_{(\flow,\m)\in \CC(\base,\marg_1,\ldots,\marg_{\kk})}\actd(\flow,\m);
\end{equation}
in particular the minima on both sides of \eqref{eq:static_dynamic_equiv} are attained. Further, 
\begin{enumerate}
\item \label{enumopt_static_dynmaic}
Any minimizer $\jointop\in \Pi(\marg_1,\ldots,\marg_{\kk})$ of the left-hand side of \eqref{eq:static_dynamic_equiv} induces a flow $(\flowop,\mop)\in \CC(\base,\marg_1,\ldots,\marg_{\kk})$ which is a minimizer of the right-hand side of \eqref{eq:static_dynamic_equiv}, satisfying $\flowop(1,\cdot)=\jointop$.
\item \label{enumopt_dynmaic_static} Any minimizer  $(\flowop,\mop)\in \CC(\base,\marg_1,\ldots,\marg_{\kk})$ of the right-hand side of \eqref{eq:static_dynamic_equiv} has the property that $\flowop(1,\cdot)$ is a minimizer of the left-hand side of \eqref{eq:static_dynamic_equiv}.
\end{enumerate}
\end{theorem}

The following result is an immediate corollary of Theorem \ref{thm:static_dynamic_equiv}.
\begin{corollary}
\label{cor:static_dynamic_equiv}
Let $\domainot\subseteq\R$ be a compact convex set, and let $\marg_1,\ldots,\marg_{\kk}$ be probability measures on $\domainot$.  Suppose that one of the following holds.
\begin{enumerate}[(i)]
\item  \label{enum:delta}$\base=\delta_0$, where we assume  $0\in\domainot^{\kk}$, and $ \cost:\R^{\kk\dd}\to \R$ is  convex.
\item \label{enum:tran_inv}   $\cost:\R^{\kk\dd}\to \R$ is a convex function such that
\begin{equation}  
\label{eq:p_trans_inv_ex_full}
\cost(x_1-\xi,\ldots,x_{\kk}-\xi)=\cost(x_1,\ldots,x_{\kk})\quad \quad\textnormal{for all}\quad x_1,\ldots,x_k\in \domainot\quad\textnormal{and}\quad \xi\in \R^{\dd},
\end{equation}
and $\base$ is of the form
\begin{equation}
\label{eq:base_Z}
\base=\law(\xi,\ldots,\xi),\qquad \xi\sim \qm\quad\textnormal{for some probability measure $\qm$ on $\domainot$}.
\end{equation}
\item \label{enum:proj} Let $E\subseteq \R^{\kk\dd}$ be a subspace and let $\proj_E$ be the orthogonal projection onto $E$. Suppose that  $\cost:\R^{\kk\dd}\to \R$ is of the form
\[
\cost(x)=\tilde{\cost}(\proj_E(x))\qquad\textnormal{for a convex function $\tilde{\cost}:E\to \R$},
\]
and $\base$ is a probability measure  on $\domainot^{\kk}$ such that $\supp(\base)\subseteq E^{\perp}$.
\end{enumerate}
In each of the cases \eqref{enum:delta},\eqref{enum:tran_inv}, \eqref{enum:proj} we have
\begin{equation}
\label{eq:static_dynamic_equiv_cor}
\min_{\joint\in \Pi(\marg_1,\ldots,\marg_{\kk})}\acts(\joint)=\min_{(\flow,\m)\in \CC(\base,\marg_1,\ldots,\marg_{\kk})}\actd(\flow,\m),
\end{equation}
and the minimum on both sides of \eqref{eq:static_dynamic_equiv_cor} is attained. Further, the conclusion of Theorem \ref{thm:static_dynamic_equiv}\eqref{enumopt_static_dynmaic}-\eqref{enumopt_dynmaic_static} hold for each of the cases \eqref{enum:delta},  \eqref{enum:tran_inv}, \eqref{enum:proj}. 
\end{corollary}
From Corollary \ref{cor:static_dynamic_equiv} we can deduce the semi-convex case when $\base=\delta_0$. 

\begin{corollary}
\label{cor:static_dynamic_equiv_semicvx}
Let $\domainot\subseteq\R$ be a compact convex set, and let $\marg_1,\ldots,\marg_{\kk}$ be probability measures on $\domainot$.  Suppose that  $0\in\domainot^{\kk}$, and that $ \cost:\R^{\kk\dd}\to \R$ is $\scvx$-semi-convex,  
\begin{equation}
\label{eq:def_semicvx}
\textnormal{there exists a constat $\scvx\ge 0$ such that $\R^{\kk\dd}\ni x\mapsto \cost(x)+\scvx\frac{|x|^2}{2}$ is convex}.
\end{equation}
Then, 
\begin{equation}
\label{eq:static_dynamic_equiv_cor_semicvx}
\begin{split}
&\min_{\joint\in \Pi(\marg_1,\ldots,\marg_{\kk})}\acts(\joint)=\\
&\min_{(\flow,\m)\in \CC(\delta_0,\marg_1,\ldots,\marg_{\kk})}\int_0^1\int_{\domainot^{\kk}}\left[\cost\left(\frac{\diff \m(t,x)}{\diff\flow(t, x)}\right)+\frac{\scvx}{2}\left|\frac{\diff \m(t,x)}{\diff\flow(t, x)}\right|^2\right]\diff\flow(t, x)\diff t-\frac{\scvx}{2}\sum_{l=1}^{\kk}\int_{\domainot} |x_l|^2\diff \marg_l(x)
\end{split},
\end{equation}
and the minimum on both sides of \eqref{eq:static_dynamic_equiv_cor_semicvx} is attained. Further,
\begin{enumerate}
\item \label{enumopt_static_dynmaic_semicvx}
Any minimizer $\jointop\in \Pi(\marg_1,\ldots,\marg_{\kk})$ of the left-hand side of \eqref{eq:static_dynamic_equiv_cor_semicvx} induces a flow $(\flowop,\mop)\in \CC(\base,\marg_1,\ldots,\marg_{\kk})$ which is a minimizer of the right-hand side of  \eqref{eq:static_dynamic_equiv_cor_semicvx}, satisfying $\flowop(1,\cdot)=\jointop$.
\item \label{enumopt_dynmaic_static_semicvx} Any minimizer  $(\flowop,\mop)\in \CC(\base,\marg_1,\ldots,\marg_{\kk})$ of the right-hand side of \eqref{eq:static_dynamic_equiv_cor_semicvx} has the property that $\flowop(1,\cdot)$ is a minimizer of the left-hand side of  \eqref{eq:static_dynamic_equiv_cor_semicvx}.
\end{enumerate}
\end{corollary}

\begin{proof}
Define the cost function $\cost_{\scvx}:\R^{\kk\dd}\to \R$ by $\cost_{\scvx}(x):=\cost(x)+\scvx\frac{|x|^2}{2}$. As $\cost_{\scvx}$ is convex,  Corollary \ref{cor:static_dynamic_equiv}\eqref{enum:delta} yields
\begin{equation}
\label{eq:equiv_dym_static_bar}
\min_{\joint\in \Pi(\marg_1,\ldots,\marg_{\kk})}\int_{\domainot^{\kk}}\cost_{\scvx}(x_1,\ldots,x_{\kk})\diff  \joint(x_1,\ldots,  x_{\kk})=\min_{(\flow,\m)\in \CC(\delta_0,\marg_1,\ldots,\marg_{\kk})}\int_0^1\int_{\domainot^{\kk}}\cost_{\scvx}\left(\frac{\diff \m(t,x)}{\diff\flow(t, x)}\right)\diff\flow(t, x)\diff t.
\end{equation}
Noting that 
\[
\int_{\domainot^{\kk}}\cost_{\scvx}(x_1,\ldots,x_{\kk})\diff  \joint(x_1,\ldots,  x_{\kk})=\int_{\domainot^{\kk}}\cost(x_1,\ldots,x_{\kk})\diff  \joint(x_1,\ldots,  x_{\kk})+\frac{\scvx}{2}\sum_{l=1}^{\kk}\int_{\domainot} |x_l|^2\diff \marg_l(x)
\]
establishes \eqref{eq:static_dynamic_equiv_cor_semicvx}. To establish items (1)-(2) note that the sets of minimizers of the left-hand sides of \eqref{eq:static_dynamic_equiv_cor_semicvx} and \eqref{eq:equiv_dym_static_bar} are identical, since  $\frac{\scvx}{2}\sum_{l=1}^{\kk}\int_{\domainot} |x_l|^2\diff \marg_l(x)$ is independent of the coupling $\joint$. The result then follows from Theorem \ref{thm:static_dynamic_equiv}\eqref{enumopt_static_dynmaic}-\eqref{enumopt_dynmaic_static}.
\end{proof}

The proof of Theorem \ref{thm:static_dynamic_equiv} is  technical and requires some preliminary work. We begin in Section \ref{subsec:proof_easy} by providing a heuristic argument for Theorem \ref{thm:static_dynamic_equiv} which, while not rigorous, provides some intuition. We then move to the preliminaries required to rigorously prove Theorem \ref{thm:static_dynamic_equiv}. In  Section \ref{subsec:cvx_cost} we show that the dynamical cost is convex and deduce some useful consequences.  We continue in Section \ref{subsec:construct_flow} with the main technical tools that relate flows to static couplings. With these preliminaries in hand we prove Theorem \ref{thm:static_dynamic_equiv} in Section \ref{subsec:proof_main}.

\subsection{Heuristic argument for Theorem \ref{thm:static_dynamic_equiv}} 
\label{subsec:proof_easy}
To prove Theorem \ref{thm:static_dynamic_equiv} it suffices to show that the right-hand side of \eqref{eq:static_dynamic_equiv} is both larger and smaller than its left-hand side.\\

\noindent \textbf{Dynamic $\ge$ Static.} Let $(\flow,\m=\vel \flow)\in  \CC(\base,\marg_1,\ldots,\marg_{\kk})$. Suppose that given $x\in \R^{\kk\dd}$ the equation 
\begin{equation}
\label{eq:ode_proof}
\partial_t\diffeo(t,x)=\vel(t,\diffeo(t,x)),~\forall~t\in [0,1],\quad\quad\diffeo(0,x)=x,
\end{equation}
has a unique solution satisfying
\begin{equation}
\label{eq:ode_pushroward}
\law(\diffeo(t,X))=\flow(t,\cdot)\quad\textnormal{with}\quad X\sim \base.
\end{equation}
(This assumption holds true when the vector field $\vel$ is regular enough.) Since $\cost$ is $\base$-translation-invariant, we have 
\begin{equation}
\label{eq:trans_inv+diag}
\cost\left(\diffeo(1,x)-x\right)=\cost\left(\diffeo(1,x)\right),\qquad \textnormal{$\base$-almost-surely}.
\end{equation}
Thus,
\begin{align*}
&\int_0^1\int_{\domainot^{\kk}}\cost(\vel(t,x))\diff\flow(t,  x)\diff t\overset{\eqref{eq:ode_pushroward}}{=}\int_0^1\int_{\domainot^{\kk}}\cost(\vel(t,\diffeo(t,x)))\diff\base( x)\diff t\overset{\eqref{eq:ode_proof}}{=}\int_0^1\int_{\domainot^{\kk}}\cost(\partial_t\diffeo(t,x)) \diff \base(x)\diff t\\
&\overset{\textnormal{Jensen}}{\ge}\int_{\domainot^{\kk}}\cost\left(\int_0^1\partial_t\diffeo(t,x)\diff t\right)\diff \base( x)=\int_{\domainot^{\kk}}\cost\left(\diffeo(1,x)-x\right)\diff \base( x)\overset{\eqref{eq:trans_inv+diag}}{=}\int_{\domainot^{\kk}}\cost\left(\diffeo(1,x)\right) \diff \base(x)\\
&\ge \min_{\textnormal{joint couplings $\joint$ of } \marg_1,\ldots,\marg_{\kk}}\int_{\domainot^{\kk}}\cost(x_1,\ldots,x_{\kk})\diff \joint(x_1,\ldots,  x_{\kk}),
\end{align*}
where the last inequality follows from $\law(\diffeo(1,X))=\flow(1,\cdot)$ which is a coupling of $\marg_1,\ldots,\marg_{\kk}$ when $X\sim \base$.\\

\noindent  \textbf{Dynamic $\le$ static.} Let $\jointop$ be a coupling of $\marg_1,\ldots,\marg_{\kk}$ which is an  optimal solution to the left-hand side of  \eqref{eq:static_dynamic_equiv}. Let $T:\R^{\dd}\to \R^{\kk\dd}$ be any  transport map\footnote{Again, when the source measure is non-atomic the existence of such map is guaranteed by  \cite[Theorem A]{MR2288266}.} between $\base$ and $\jointop$, and let 
\[
\diffeo(t,x):=(1-t)x+tT(x)\quad\textnormal{for}\quad t\in [0,1]. 
\]
Suppose that $x\mapsto \diffeo(t,x)$ is invertible for all $t\in [0,1]$. Then $(\diffeo(t,x))_{t\in [0,1]}$ satisfies the equation
\begin{equation}
\label{eq:ode_proof_2}
\forall~t\in [0,1],\quad \quad T(x)-x=\partial_t\diffeo(t,x)=\vel(t,\diffeo(t,x)),
\end{equation}
with
\begin{equation}
\label{eq:vec_opt}
\vel(t,x):=T\left(\diffeo^{-1}(t,x)\right)-\diffeo^{-1}(t,x).
\end{equation}
Let
\begin{equation}
\label{eq:ode_pushroward_2}
\flow(t,\cdot):=\law(\diffeo(t,X))\quad\textnormal{where}\quad X\sim \base,
\end{equation}
and note that $(\flow,\m:=\vel\flow)\in \CC(\base,\marg_1,\ldots,\marg_{\kk})$. Since $\cost$ is $\base$-translation-invariant, we have
\begin{equation}
\label{eq:trans_inv+diag_2}
\cost\left(T(x)-x\right)=\cost\left(T(x)\right),\qquad \textnormal{$\base$-almost-surely}.
\end{equation}
Thus,
\begin{align*}
&\int_0^1\int_{\domainot^{\kk}}\cost(\vel(t,x))\diff \flow(t, x)\diff t\overset{\eqref{eq:ode_pushroward_2}}{=}\int_0^1\int_{\domainot^{\kk}}\cost(\vel(t,\diffeo(t,x))) \diff\base( x)\diff t\overset{\eqref{eq:ode_proof_2}}{=}\int_0^1\int_{\domainot^{\kk}}\cost(T(x)-x)\diff \base( x)\diff t\\
&\overset{\eqref{eq:trans_inv+diag_2}}{=}\int_{\domainot^{\kk}}\cost(T(x))\diff \base( x)=\int_{\domainot^{\kk}}\cost(x)\diff\jointop(x)\le  \min_{\textnormal{joint couplings $\joint$ of } \marg_1,\ldots,\marg_{\kk}}\int_{\domainot^{\kk}}\cost(x_1,\ldots,x_{\kk})\diff \joint( x_1,\ldots, x_{\kk}).
\end{align*}
The argument outlined above is only meant to provide intuition. We will now start building the  necessary tools for a rigorous proof of  Theorem \ref{thm:static_dynamic_equiv}. 
\subsection{Convexity of the dynamical cost}
\label{subsec:cvx_cost}
In this section we establish the convexity of the dynamical cost and its consequences. We start with the classical result relating the convexity properties of a function and its perspective \cite{combettes2018perspective}.

\begin{lemma}
\label{lem:cvx_cost}
If $\cost:\R^{\kk\dd}\to \R$ is convex then the  function 
\[
[0,\infty)\times \R^{\kk\dd}\ni (\pi,m)\mapsto \cost\left(\frac{m}{\pi}\right)\pi
\]
is convex. In particular, the function $(\flow,\m)\mapsto \actd(\flow,\m)$ is also convex.
\end{lemma}
One consequence of Lemma \ref{lem:cvx_cost} is that the dynamical cost decreases under convolution. Let us define the type of the convolution kernels that we will consider. 
\begin{definition}
\label{def:conv_ker}
A function $\reg:\R^{\kk\dd}\to \R$ is a \emph{probability kernel}
if  it is even and satisfies
\[
\forall~ x\in\R^{\kk\dd},\qquad \reg(x)\ge 0,\qquad \int_{\R^{\kk\dd}} \reg(y-x)\diff y=1, \qquad\text{and}\qquad \supp(\reg)\subseteq B_R,
\]
where $B_R$ is the ball in $\R^{\kk\dd}$ centered at the origin with radius $R$ for some $R>0$.
\end{definition}
The existence of a probability kernel as in Definition \ref{def:conv_ker} is known; cf. \cite[Appendix C.4]{MR2597943}. In particular, using the latter construction,  there exists a sequence of probability kernels $\{\reg(\ell)\}_{\ell}$ converging weakly  to a point-mass at the origin \cite[Appendix C.4, Theorem 6(iii)]{MR2597943}.

The following result shows that the dynamical cost decreases under convolution with probability kernels. 
\begin{lemma}
\label{lem:cost_reg}
Suppose that the cost function $\cost:\R^{\kk\dd}\to \R$ is convex, and fix a  probability kernel $\reg$. Then, for any $(\flow ,\m)\in C\left([0,1]; \mathcal P(\domainot^{\kk})\right)\times \mathcal M([0,1]\times\domainot^{\kk};\R^{\kk\dd})$, such that $\m\ll\flow$, we have
\begin{equation}
\label{eq:cost_reg}
\actd(\flow\ast\reg,\m\ast\reg)\le \actd(\flow,\m),
\end{equation}
where
\begin{equation}
\label{eq:extend_dyn_cost_def}
\actd(\flow\ast\reg,\m\ast\reg):=\int_0^1\int_{\R^{\kk\dd}} \cost\left(\frac{(\m\ast\reg)(t,x)}{(\flow\ast\reg)(t,x)}\right)(\flow\ast \reg)(t,x)\diff x\diff t.
\end{equation}
\end{lemma}

\begin{proof}
Fix $x\in \R^{\kk\dd}$ and let $\vel(t,y):=\frac{\diff  \m(t,y)}{\diff \flow(t,y)}$ for $(t,y)\in [0,1]\times\domainot^{\kk}$. Applying Jensen's inequality with the measure $\flow(t,\cdot)$, and the convex function (see Lemma \ref{lem:cvx_cost}), 
\[
 \R^{\kk\dd}\times [0,\infty)\ni  (\reg\vel,\reg)\mapsto\cost\left(\frac{\reg \vel}{\reg}\right)\reg= \cost(\vel)\reg,
\] 
we get
\begin{equation}
\label{eq:jensen_kappa}
\begin{split}
&\cost\left(\frac{ (\m\ast\reg)(t,x)}{ (\flow\ast\reg)(t,x)}\right) (\flow\ast \reg)(t,x)=\cost\left(\frac{\int_{\R^{\kk\dd}}\reg(y-x) \diff  \m(t,y) }{\int_{\R^{\kk\dd}} \reg(y-x)\diff \flow(t,y)}\right)\int_{\R^{\kk\dd}}\reg(y-x)  \diff\flow(t, y)\\
&=\cost\left(\frac{\int_{\R^{\kk\dd}}\reg(y-x) \vel(t,y)\diff  \flow(t,y) }{\int_{\R^{\kk\dd}} \reg(y-x)\diff \flow(t,y)}\right)\int_{\R^{\kk\dd}}\reg(y-x)  \diff\flow(t, y)\le \int_{\R^{\kk\dd}} \left[\cost\left(\frac{\reg(y-x)\vel(t,y)}{\reg(y-x)}\right)\reg(y-x)\right] \diff\flow(t, y)\\
&=\int_{\R^{\kk\dd}} \left[\cost\left(\frac{\diff  \m(t,y)}{\diff \flow(t,y)}\right)\reg(y-x)\right] \diff\flow(t, y).
\end{split}
\end{equation}
Integrating \eqref{eq:jensen_kappa} against $\diff x$ and $\diff t$, and recalling that  $\int\reg(y-x)\diff x=1$, we get
\begin{align*}
&\int_0^1\int_{\R^{\kk\dd}}  \cost\left(\frac{( \m\ast\reg)(t,x)}{(\flow\ast\reg)(t,x)}\right) (\flow\ast \reg)(t,x)\diff x\diff t\le\int_0^1\left\{ \left[\int_{\R^{\kk\dd}} \cost\left(\frac{  \diff\m(t, y)}{\diff \flow(t, y)}\right)\diff \flow(t,y)\right] \int_{\R^{\kk\dd}} \reg(y-x)\diff x\right\}\diff t\\
&=\int_0^1 \int_{\R^{\kk\dd}} \cost\left(\frac{  \diff\m(t, y)}{\diff \flow(t, y)}\right)\diff \flow(t,y) \diff t,
\end{align*}
establishing \eqref{eq:cost_reg}. 
\end{proof}

\begin{remark}
Lemma \ref{lem:cost_reg} and its proof follow \cite[Lemma 8.1.10]{ambrosio2008gradient}, but we need to require $\reg$ to be a probability kernel, rather than any convolution kernel, because our cost $\cost$ need not be 1-homogeneous. 
\end{remark}

\subsection{Constructions of flows}
\label{subsec:construct_flow}
In this section we show how to relate transport plans in the static problem to flows in the dynamical problem, and vice versa. If we inspect the argument in Section \ref{subsec:proof_easy} we see that there are two assumptions that need to be justified. 
\begin{enumerate}
\item Static $\ge$ dynamic. Justify that given $(\flow,\m)$ the equation \eqref{eq:ode_proof} has a unique solution satisfying \eqref{eq:ode_pushroward}.
\item Static $\le$ dynamic. The invertibility of $\diffeo(t,x):=(1-t)x+tT(x)$ for all $t\in [0,1]$, where $T$ transports $\base$ to $\jointop$.
\end{enumerate}
If we do not assume sufficient regularity then both assumptions need not  be true. To address (1) we use the probabilistic representation approach of Ambrosio \cite{MR2096794}, while for (2) we use  an argument of Brenier \cite{MR2006306}.

Let us begin with the preliminaries for (1). The probabilistic representation approach shows that while \eqref{eq:ode_proof} may not have a \emph{unique} solution, there exists a probability measure over possible solutions of \eqref{eq:ode_proof} which can be used as a substitute. 
\begin{proposition}
\label{prop:thm8.2.1AGS}
Let  $\domainot\subseteq \R^{\dd}$ be a compact convex set, $\base$ a probability measure over $\domainot^{\kk}$, and $\marg_1,\ldots,\marg_{\kk}$ probability measures over $\domainot$. Fix $(\flow,\m)\in \CC(\base,\marg_1,\ldots,\marg_{\kk})$. Then, there exists a probability measure $\mathbb P$ on $\domainot^{\kk}\times C([0,1];\domainot^{\kk})$ which is concentrated on the set of pairs $(x,\diffeo)$, where
\begin{equation}
\label{eq:ode_proof_ambrosio}
\partial_t\diffeo(t,x)=\vel(t,\diffeo(t,x)),~\textnormal{for almost-everywhere}~t\in (0,1),\quad\quad\diffeo(0,x)=x,
\end{equation}
and for any continuous and bounded function $\test:\domainot^{\kk}\to \R$,
\begin{equation}
\label{eq:eq8.2.1AGS}
\int_{\domainot^{\kk}}\test(x) \diff\flow(t, x)=\int_{\domainot^{\kk}\times C([0,1];\domainot^{\kk})}\test(\diffeo(t,x)) \diff \mathbb P(x, \diffeo)\quad\quad\textnormal{for all }t\in [0,1].
\end{equation}
\end{proposition}

\begin{proof}
By \cite[Theorem 8.2.1]{ambrosio2008gradient} the theorem holds true if $\domainot=\R^{\dd}$. Thus, if we interpret $(\flow,\m)$ as measures on $\R^{\kk\dd}$, then we know that there exists a probability measure $\mathbb P$ on $\R^{\kk\dd}\times C([0,1];\R^{\kk\dd})$ such that $(x,\diffeo)\sim \mathbb P$ satisfies \eqref{eq:ode_proof_ambrosio}, and for any continuous and bounded function $\test:\R^{\kk\dd}\to \R$,
\begin{equation}
\label{eq:eq8.2.1AGS_Euclid}
\int_{\R^{\kk\dd}}\test(x) \diff \flow(t, x)=\int_{\R^{\kk\dd}\times C([0,1];\R^{\kk\dd})}\test(\diffeo(t,x)) \diff\mathbb P(x, \diffeo)\quad\quad\textnormal{for all }t\in [0,1].
\end{equation}
Applying \eqref{eq:eq8.2.1AGS_Euclid} with $t=0$, we can conclude that $\mathbb P$ is supported on $\domainot^{\kk}\times C([0,1];\R^{\kk\dd})$, since the marginal of $x$ under $\mathbb P$ is $\base$, which by assumption is supported on $\domainot^{\kk}$. Thus, it suffices to show that 
\begin{equation}
\label{eq:support_on_D}
\mathbb P(\{(x,\diffeo):\diffeo(t,x)\notin \domainot^{\kk} \textnormal{ for some $t\in [0,1]$ and $x\in \domainot^{\kk}$} \})=0.
\end{equation}
Applying \eqref{eq:eq8.2.1AGS_Euclid} for a fixed $t\in [0,1]$ we get that 
\begin{equation}
\label{eq:support_on_D_t}
\mathbb P(\{(x,\diffeo):\diffeo(t,x)\notin \domainot^{\kk} \textnormal{ for some $x\in \domainot^{\kk}$}\})=0,
\end{equation}
since $\flow(t,\cdot)$ is supported on $\domainot^{\kk}$. Hence, it remains to show that the negligible sets for each $t\in [0,1]$ do not pile up. Let $\mathbb Q_{[0,1]}$ be the set of rational numbers in $[0,1]$. Then, 
\begin{align*}
&\mathbb P(\{(x,\diffeo):\diffeo(t,x)\notin \domainot^{\kk} \textnormal{ for some $t\in \mathbb Q_{[0,1]}$ and $x\in \domainot^{\kk}$} \})\\
&\le \sum_{t\in \mathbb Q_{[0,1]}}\mathbb P(\{(x,\diffeo):\diffeo(t,x)\notin \domainot^{\kk} \textnormal{ for some $x\in \domainot^{\kk}$}\})=0.
\end{align*}
Since $\domainot^{\kk}$ is closed, and $\mathbb Q_{[0,1]}$ is dense in $[0,1]$, we have that $\diffeo(t,x)\in \domainot^{\kk}$ for all $t\in [0,1]$ and $x\in  \domainot^{\kk}$, if and only if, $\diffeo(t,x)\in \domainot^{\kk}$ for all $t\in \mathbb Q_{[0,1]}$ and $x\in  \domainot^{\kk}$. The latter has measure zero which proves \eqref{eq:support_on_D}. 
\end{proof}
Next we move to the preliminaries necessary for addressing (2). Rather than using a transport map $T$ from $\base$ to $\jointop$, we interpolate in law between $\base$ and $\jointop$, and then construct a matching vector field to mimic this interpolation. 
\begin{proposition}
\label{prop:velocity_exist}
Let $\domainot\subseteq \R^{\dd}$ be a compact convex set. Fix  $q_a,q_b\in \pspace(\domainot^{\kk})$ and $q_{ab}\in\Pi(q_a,q_b)$. Let
\begin{equation}
\label{eq:flow_opt_def}
\flow(t,\cdot):=\law(X(t,Y^a,Y^b))\quad\textnormal{with}\quad (Y^a,Y^b)\sim q_{ab},
\end{equation}
and $X:[0,1]\times \domainot^{\kk}\times \domainot^{\kk}\to\R^{\kk\dd}$ given by
 \begin{equation}
\label{eq:linear_interp_def}
X(t,y^a,y^b):=(1-t)y^a+ty^b.
\end{equation} 
Given $t\in [0,1]$ and $x\in \domainot$ disintegrate\footnote{Colloquially, $q_{ab}^{t,x}=\law((Y^a,Y^b)|X(t,Y^a,Y^b)=x)$ where $(Y^a,Y^b)\sim  q_{ab}$.} $\diff q_{ab}(y^a, y^b)=\diff q_{ab}^{t,x}( y^a,y^b)\otimes \diff \flow(t, x)$ so that, for every continuous and bounded function $\test:\domainot^{\kk}\times \domainot^{\kk}\times \domainot^{\kk}\to \R$, we have
\begin{equation}
\label{eq:gamma_cond}
\int_{\domainot^{\kk}\times \domainot^{\kk}}\test(y^a,y^b,(1-t)y^a+ty^b)\diff q_{ab}( y^a, y^b)=\int_{\domainot^{\kk}}\int_{\domainot^{\kk}\times \domainot^{\kk}}\test(y^a,y^b, x)\diff q_{ab}^{t,x}( y^a, y^b)\diff\flow(t, x).
\end{equation}
Define the vector-valued measure $\m(t,\cdot)$ as being absolutely continuous with respect to $\flow(t,\cdot)$, with the vector-valued density $\vel:[0,1]\times\domainot^{\kk}\to\R^{\kk\dd}$ given by
\begin{equation}
\label{eq:vec_coupling}
\vel(t,x):=\int_{\domainot^{\kk}\times\domainot^{\kk} }(y^b-y^a) \diff  q_{ab}^{t,x}(y^a, y^b).
\end{equation}
Then, $(\flow,\m)$ solves the continuity equation $\partial_t\flow+\nabla\m=0$ in a weak sense. 
\end{proposition}
\begin{proof}
We need to show that for all $\test\in C^1([0,1]\times\domainot^{\kk})$,
\begin{equation}
\label{eq:weak_cont_proof}
\int_0^1\int_{\domainot^{\kk}}\partial_t\test(t,x)\diff \flow(t,x)\diff t+\int_0^1\int_{\domainot^{\kk}}\langle\vel(t,x),\nabla\test(t,x) \rangle\diff \flow(t,x)\diff t-\int_{\domainot^{\kk}}\test(1,x) \diff  q_b(x)+\int_{\domainot^{\kk}}\test(0,x) \diff q_a( x)=0.
\end{equation}
To verify \eqref{eq:weak_cont_proof} we note that the first integral in \eqref{eq:weak_cont_proof} reads
\begin{align*}
\int_0^1\int_{\domainot^{\kk}}\partial_t\test(t,x) \diff\flow(t,x)\diff t\overset{\eqref{eq:flow_opt_def}}{=}\int_0^1\int_{\domainot^{\kk}\times \domainot^{\kk}}\partial_t\test(t,(1-t)y^a+ty^b)\diff q_{ab}( y^a, y^b)\diff t,
\end{align*}
while the second integral in \eqref{eq:weak_cont_proof} reads
\begin{align*}
&\int_0^1\int_{\domainot^{\kk}}\langle\vel(t,x),\nabla\test(t,x)\rangle \diff\flow(t, x)\diff t\overset{\eqref{eq:vec_coupling}}{=}\int_0^1\int_{\domainot^{\kk}}\left\langle\int_{\domainot^{\kk}\times \domainot^{\kk} }(y^b-y^a) \diff q_{ab}^{t,x}( y^a, y^b),\nabla\test(t,x)\right\rangle\diff \flow(t, x)\diff t\\
&\overset{\eqref{eq:gamma_cond}}{=}\int_0^1\int_{\domainot^{\kk}\times \domainot^{\kk}}\langle y^b-y^a,\nabla\test(t,(1-t)y^a+ty^b) \diff q_{ab}( y^a, y^b)\diff t.
\end{align*}
Thus, the sum of the first two integrals  in \eqref{eq:weak_cont_proof} can be written as
\begin{align*}
&\int_0^1\int_{\domainot^{\kk}\times \domainot^{\kk}}[\partial_t\test(t,(1-t)y^a+ty^b)+\langle(y^b-y^a),\nabla\test(t,(1-t)y^a+ty^b)\rangle] \diff  q_{ab}(y^a,y^b)\diff t\\
&=\int_0^1\int_{\domainot^{\kk}\times \domainot^{\kk}}\frac{\diff}{\diff t}\test(t,(1-t)y^a+ty^b)\diff  q_{ab}(y^a, y^b)\diff t=\int_{\domainot^{\kk}\times \domainot^{\kk}}[\test(1,y^b)-\test(0,y^a)]\diff q_{ab}( y^a, y^b)\\
&=\int_{\domainot^{\kk}}\test(1,y^b) \diff q_b( y^b)-\int_{\domainot^{\kk}}\test(0,y^a) \diff  q_a(y^a),
\end{align*}
where the last identity used $q_{ab}\in\Pi(q_a,q_b)$. 
\end{proof}

\begin{remark}
\label{rem:neumann}
Implicit in the proof of Proposition \ref{prop:velocity_exist} is that the vector field defined by \eqref{eq:vec_coupling} satisfies the Neumann boundary condition. The intuition behind this fact is that the vector field along the trajectories \eqref{eq:linear_interp_def} always sends mass into $\domainot$, unless both $y^a,y^b\in \partial \domainot$, in which case the mass is transported parallel to the boundary $\partial \domainot$, and hence  $\vel$ satisfies the Neumann boundary condition.
\end{remark}

\begin{remark}
It is instructive to consider Proposition \ref{prop:thm8.2.1AGS} and Proposition \ref{prop:velocity_exist} in the case where $\base=\delta_0$. Starting with Proposition \ref{prop:velocity_exist}, in this setting we have $q^{t,x}_{ab}(y_a,y_b)=\delta_{\frac{x}{t}}(y_b)$, and hence $\vel(t,x)=\frac{x}{t}$. The solution of the continuity equation $\partial_t\flow+\nabla \m=0$ is $\flow(t,\cdot)=\law(tY_b)$, where $Y_b\sim \jointop$, which indeed interpolates between $\delta_0$ at $t=0$ and $\flowop$ at $t=1$. Further, while the vector field $\vel(t,x)$ diverges as $t\downarrow 0$, the norm of the vector field is finite as 
\[
\int_0^1\int_{\domainot^{\kk}}\left|\frac{\diff\m(t, x)}{\diff\flow(t, x)}\right|^2\diff\flow(t, x)\diff t=\int_0^1\int_{\domainot^{\kk}}\left|\frac{tx}{t}\right|^2\diff\flowop (x)\diff t=\int_{\domainot^{\kk}}|x|^2\diff\flowop (x)<+\infty. 
\]
Looking at the trajectories of the individual particles,
\begin{equation}
\label{eq:lagrange_delta}
\partial_t\diffeo(t,x)=\vel(t,\diffeo(t,x))=\frac{\diffeo(t,x)}{t},
\end{equation}
we find that,  for some $x_0\in\domainot^{\kk}$, $\diffeo(t,x) = t x_0$ for all $x\in \domainot^{\kk}$. In particular, we see that $x_0$ is not uniquely defined since $\diffeo(0,x)=0$ regardless of $x_0$. Indeed, the vector field $\vel(t,x)=\frac{x}{t}$ is not Lipschitz, so \eqref{eq:lagrange_delta} does not have a unique solution. Rather, one has a distribution, as in Proposition \ref{prop:thm8.2.1AGS}, over possible trajectories. These trajectories are generated by first sampling $x_0\sim \jointop$, and then setting $\diffeo(t,x):=t x_0$ for all $x\in \domainot^{\kk}$. 
\end{remark}

\subsection{Proof of Theorem \ref{thm:static_dynamic_equiv}}
\label{subsec:proof_main}
We are now ready for the proof of Theorem \ref{thm:static_dynamic_equiv}.  We start by noting that  $\inf_{\joint\in \Pi(\marg_1,\ldots,\marg_{\kk})}\acts(\joint)=\min_{\joint\in \Pi(\marg_1,\ldots,\marg_{\kk})}\acts(\joint)$ since, with respect to the weak convergence topology, $\joint\mapsto \acts(\joint)$ is continuous (by the continuity of $\cost$), and $\Pi(\marg_1,\ldots,\marg_{\kk})$ is compact (by the boundedness of $\domainot$)\footnote{This argument is standard, e.g., see \cite[Theorem 1.4]{Santambrogio_book} for the argument in the two-marginal case.}. As in Section \ref{subsec:proof_easy}, we show that the right-hand side of \eqref{eq:static_dynamic_equiv} is both larger and smaller than its left-hand side. \\

\noindent \textbf{Dynamic $\ge$ static.} Fix $(\flow,\m)\in\CC(\base,\marg_1,\ldots,\marg_{\kk})$, and a probability kernel $\reg$. Denote $\flow^{\reg}:=\flow\ast\reg$, $\m^{\reg}:=\m\ast\reg$, and let $\vel^{\reg}$ be the density of $\m^{\reg}$ with respect to $\flow^{\reg}$. Note that $\vel^{\reg}$ is continuous as it is the ratio of two continuous functions. We denote $\base^{\reg}:=\base\ast\reg$, and for $l\in [\kk]$, denote $\marg_l^{\reg}:=\marg_l\ast\reg$. We let $\domainot_{\reg}^{\kk}$ be the domain of $\flow^{\reg}$, where  $\domainot_{\reg} \subseteq \R^{\dd}$ is some compact set containing $\domainot$.

\begin{claim} 
The following hold:
\begin{align}
&(\flow^{\reg} ,\m^{\reg})\in C\left([0,1]; \mathcal P(\domainot_{\reg}^{\kk})\right)\times \mathcal M([0,1]\times\domainot_{\reg}^{\kk};\R^{\kk\dd}); \m^{\reg}\ll\flow^{\reg},\label{eq:ker_domain_flow}\\
&\int_0^1\int_{\domainot^{\kk}}\left|\frac{\diff\m^{\reg}(t, x)}{\diff\flow^{\reg}(t, x)}\right|^p\diff\flow^{\reg}(t, x)\diff t<+\infty\quad\textnormal{for some $p>1$,}\label{eq:ker_integ_condition}\\
&\partial_t\flow^{\reg}+\nabla \m^{\reg}=0\textnormal{ in a weak sense},\label{eq:ker_cont_eq}\\
&\flow^{\reg}(0, \cdot)=\base^{\reg}\,\textnormal{ and } \,\flow_l^{\reg}(1,\cdot)=\marg_l^{\reg}\textnormal{ for all }l\in [\kk]\label{eq:ker_marg}.
\end{align}
\end{claim}
\begin{proof}
The assertions $(\flow^{\reg} ,\m^{\reg})\in C\left([0,1]; \mathcal P(\domainot_{\reg}^{\kk})\right)\times \mathcal M([0,1]\times\domainot_{\reg}^{\kk};\R^{\kk\dd})$, and $ \m^{\reg}\ll\flow^{\reg}$, hold as the convolution smooths out the measures. The integrability condition holds by the compactness of $\domainot_{\reg}^{\kk}$ and since, as noted above, $\vel^{\reg}$ is continuous.  That $(\flow^{\reg} ,\m^{\reg})$ satisfy the continuity equation in a weak sense follows as, for any $\test\in C^1([0,1]\times\domainot_{\reg}^{\kk}; \R)$, 
\begin{align}
\label{eq:cont_eq_convolution}
\begin{split}
&\int_0^1\int_{\domainot_{\reg}^{\kk}}\partial_t\test(t,x)\diff\flow^{\reg}(t, x)\diff t+\int_0^1\int_{\domainot_{\reg}^{\kk}}\langle\nabla\test(t,x),\diff\m^{\reg}(t, x)\rangle \diff t\\
&=\int_0^1\int_{\domainot^{\kk}}\left(\int_{\R^{\kk\dd}}\partial_t\test(t,x) \reg(y-x) \diff x\right)\diff \flow(t,y)\diff t+\int_0^1\int_{\domainot^{\kk}}\left\langle\int_{\R^{\kk\dd}}\reg(y-x)\nabla\test(t,x)\diff x, \diff  \m(t,y)\right\rangle \diff t\\
&=\int_0^1\int_{\domainot^{\kk}}\partial_t(\test\ast \reg)(t,y) \diff \flow(t,y)\diff t+\int_0^1\int_{\domainot^{\kk}}\left\langle\nabla(\test\ast\reg)(t,y), \diff  \m(t,y)\right\rangle \diff t\\
&=\int_{\domainot^{\kk}}(\eta\ast\reg)(1,x)\diff\flow(1, x)-\int_{\domainot^{\kk}}(\eta\ast\reg)(0,x)\diff\flow(0, x)\\
&=\int_{\domainot^{\kk}}\int_{\R^{\kk\dd}}\reg(y-x)\test(1,y)\diff\flow(1, x)-\int_{\domainot^{\kk}}\int_{\R^{\kk\dd}}\reg(y-x)\test(0,x)\diff\flow(0, x)\\
&=\int_{\domainot_{\reg}^{\kk}}\test(1,y)\diff \flow^{\reg}(1,y)-\int_{\domainot_{\reg}^{\kk}}\test(0,y)\diff \flow^{\reg}(0,y),
\end{split}
\end{align}
where we used $\test\ast\reg\in C^1([0,1]\times\domainot^{\kk}; \R)$, and that $(\flow,\m)$ satisfies the continuity equation. Lastly, $\flow^{\reg}(0, \cdot)=\base^{\reg}$, and $\flow_l^{\reg}(1,\cdot)=\marg_l^{\reg}$ for all $l\in [\kk]$, hold by construction.
\end{proof}
Below we apply Proposition \ref{prop:thm8.2.1AGS} with  the domain $\domainot_{\reg}$, flow  $(\flow^{\reg},\m^{\reg})$, and the associated $\mathbb P^{\reg}$.  We write $\diffeo^{\reg}$ for solutions of \eqref{eq:ode_proof_ambrosio} with $\vel^{\reg}$. 
The function $\cost\circ\vel^{\reg}$ is continuous and bounded, since $\vel^{\reg}$ and $\cost$ are both continuous, and since $\vel^{\reg}$  has bounded support. Hence, applying \eqref{eq:eq8.2.1AGS} with $\test(x):=\cost(\vel^{\reg}(t,x))$, for any fixed $t\in [0,1]$, we get
\begin{align}
\begin{split}
\label{eq:bigger_than}
&\int_0^1\int_{\domainot^{\kk}}\cost(\vel(t,x))\diff\flow(t, x)\diff t\overset{\eqref{eq:cost_reg}}{\ge}\int_0^1\int_{\domainot_{\reg}^{\kk}}\cost(\vel^{\reg}(t,x))\diff \flow^{\reg}(t,x)\diff t\\
&\overset{\eqref{eq:eq8.2.1AGS}}{=}\int_0^1\int_{\domainot_{\reg}^{\kk}\times C([0,1];\domainot_{\reg}^{\kk})}\cost(\vel^{\reg}(t,\diffeo^{\reg}(t,x)))\diff \mathbb P^{\reg}(  x, \diffeo^{\reg})\diff t\\
&\overset{\eqref{eq:ode_proof_ambrosio}}{=}\int_0^1\int_{\domainot_{\reg}^{\kk}\times C([0,1];\domainot_{\reg}^{\kk})}\cost(\partial_t\diffeo^{\reg}(t,x))\diff \mathbb P^{\reg}(  x, \diffeo^{\reg})\diff t\overset{\textnormal{Jensen}}{\ge}\int_{\domainot_{\reg}^{\kk}\times C([0,1];\domainot_{\reg}^{\kk})}\cost\left(\int_0^1\partial_t\diffeo^{\reg}(t,x)\diff t\right)\diff\mathbb P^{\reg}(  x,\diffeo^{\reg})\\
&=\int_{\domainot_{\reg}^{\kk}\times C([0,1];\domainot_{\reg}^{\kk})}\cost\left(\diffeo^{\reg}(1,x)-x\right)\diff \mathbb P^{\reg}( x, \diffeo^{\reg}).
\end{split}
\end{align}
The assumption that $\cost$ is $\base$-translation-invariant does not imply that $\cost$ is $\base^{\reg}$-translation-invariant, so we cannot immediately argue in \eqref{eq:bigger_than} as we did in Section \ref{subsec:proof_easy}. The following argument circumvents this issue. Let $\{\reg(\ell)\}_{\ell}$ be a sequence of probability kernels converging weakly  to a point-mass at the origin, and let $\tilde{\domainot}^{\kk}$ be a compact set containing the support of $\flow^{\reg(\ell)}$ for all $\ell$. By \eqref{eq:bigger_than}, for every $\ell$,
\begin{equation}
\label{eq:bigger_than_ell}
\int_0^1\int_{\domainot^{\kk}}\cost(\vel(t,x))\diff \flow(t,x)\diff t\ge \int_{\tilde{\domainot}^{\kk}\times C([0,1];\tilde{\domainot}^{\kk})}\cost\left(\diffeo^{\reg(\ell)}(1,x)-x\right)\diff \mathbb P^{\reg(\ell)}( x,\diffeo^{\reg(\ell)}).
\end{equation}
Next we take the limit $\ell\to\infty$ on the right-hand side of \eqref{eq:bigger_than_ell}. To this end we note that the right-hand side of \eqref{eq:bigger_than_ell} depends only on the joint law of $\mathbb P^{\reg(\ell)}$ over $x$ and $\diffeo^{\reg(\ell)}(1,\cdot)$, which is a measure over $\tilde{\domainot}^{\kk}\times \tilde{\domainot}^{\kk}$. By compactness, there exists a converging subsequence with limit $\joint^{\infty}$ which is a probability measure over $\tilde{\domainot}^{\kk}\times \tilde{\domainot}^{\kk}$. Further, applying \eqref{eq:eq8.2.1AGS} for each $\ell$ at $t=0$ and $t=1$, and taking the limit $\ell\to\infty$, we conclude that the first marginal of  $\joint^{\infty}$ on  $\tilde{\domainot}^{\kk}$ is $\base$, while the second marginal of  $\joint^{\infty}$ on  $\tilde{\domainot}^{\kk}$ has marginals $\marg_1,\ldots,\marg_{\kk}$ on $\domainot$. In particular, we may replace $\tilde{\domainot}^{\kk}$ by $\domainot^{\kk}$. Hence, we conclude from \eqref{eq:bigger_than_ell} that
\begin{align}
\begin{split}
\label{eq:bigger_than_infty}
\int_0^1\int_{\domainot^{\kk}}\cost(\vel(t,x))\diff\flow(t, x)\diff t&\ge \int_{\domainot^{\kk}\times \domainot^{\kk}}\cost(y-x)\diff \joint^{\infty}(x,y)\overset{\eqref{eq:trans_inv_cost_thm}}{=}\int_{\domainot^{\kk}\times \domainot^{\kk}}\cost(y)\diff \joint^{\infty}(x, y)\\
&\ge \min_{\joint\in \Pi(\marg_1,\ldots,\marg_{\kk})}\acts(\joint),
\end{split}
\end{align}
so it follows that
\begin{equation}
\label{eq:dyn_stat_big_proof}
\inf_{(\flow,\m)\in \CC(\base,\marg_1,\ldots,\marg_{\kk})}\actd(\flow,\m)\ge\min_{\joint\in \Pi(\marg_1,\ldots,\marg_{\kk})}\acts(\joint),
\end{equation}
as claimed.

\noindent \textbf{Dynamic $\le$ static.}

Let $\jointop\in \Pi(\marg_1,\ldots,\marg_{\kk})$ be an optimal solution to the left-hand side of \eqref{eq:static_dynamic_equiv}. Set $q_a:=p$, $q_b:=\jointop$, and choose any $q_{ab}\in \Pi(q_a,q_b)$. Define the flow $\flowop$ by \eqref{eq:flow_opt_def}, and define $ \mop$ as the vector-valued measure whose vector-valued density $\velop$ with respect to $\flowop$ is given by  \eqref{eq:vec_coupling}. We first need to show that  $(\flowop,\mop)$ is a valid flow. 

\begin{claim} 
$(\flowop,\mop)\in\CC(\base,\marg_1,\ldots,\marg_{\kk})$. 
\end{claim}
\begin{proof}
First we show that $\flowop\in C\left([0,1]; \mathcal P(\domainot^{\kk})\right)$. Indeed, given a continuous and bounded function $\test:\domainot^{\kk}\to \R$ we have
\[
\int_{\domainot^{\kk}}\eta(x)\diff \flowop(t,x)\overset{\eqref{eq:flow_opt_def}}{=}\int_{\domainot^{\kk}\times \domainot^{\kk}}\eta((1-t)y^a+ty^b)\diff q_{ab}(y^a,y^b),
\]
which is a continuous function in $t$. By definition, $\mop\in \mathcal M([0,1]\times\domainot^{\kk}; \R^{\kk\dd})$, so we verified $(\flowop ,\mop)\in C\left([0,1]; \mathcal P(\domainot^{\kk})\right)\times \mathcal M([0,1]\times\domainot^{\kk}; \R^{\kk\dd})$. The integrability condition $\int_0^1\int_{\domainot^{\kk}}\left|\frac{\diff\m}{\diff\flow}\right|^p\diff\flow\diff t<+\infty$ holds since, by the compactness of $\domainot$, the vector field $\vel(t,x)$ defined by \eqref{eq:gamma_cond} is bounded uniformly by a constant for all $t\in [0,1]$ and $x\in\R^{\kk\dd}$. That $(\flowop, \mop)$ satisfies the continuity equation holds by Proposition \ref{prop:velocity_exist}. The fact that $\flowop(0,\cdot)=\base$, and $(\flowop)_l(1,\cdot)=\marg_l\textnormal{ for all }l\in [\kk]$, hold by construction. 
\end{proof}
We now show that $\actd(\flowop,\mop)\le \acts(\jointop)$. Indeed,
\begin{align}
\begin{split}
\label{eq:smaller_than}
&\actd(\flowop,\mop)=\\
&\int_0^1\int_{\domainot^{\kk}}\cost(\velop(t,x))\diff \flowop(t,x)\diff t\overset{\eqref{eq:vec_coupling}}{=}\int_0^1\int_{\domainot^{\kk}}\cost\left(\int_{\domainot^{\kk}\times\domainot^{\kk}}(y^b-y^a)\diff q_{ab}^{t,x}(y^a, y^b)\right)\diff\flowop(t, x)\diff t\\
&\overset{\textnormal{Jensen}}{\le}\int_0^1\int_{\domainot^{\kk}}\int_{\domainot^{\kk}\times\domainot^{\kk}}\cost(y^b-y^a) \diff q_{ab}^{t,x}(y^a, y^b)\diff \flowop(t, x)\diff t=\int_{\domainot^{\kk}\times\domainot^{\kk}}\cost(y^b-y^a) \diff q_{ab}(y^a,y^b)\\
&\overset{\eqref{eq:trans_inv_cost_thm}}{=}\int_{\domainot^{\kk}}\cost(y^b) \diff q^b( y^b)=\int_{\domainot^{\kk}}\cost(x_1,\ldots,x_k)\diff \jointop(x_1,\ldots,x_k)=\acts(\jointop).
\end{split}
\end{align}
Thus,
\begin{equation}
\label{eq:dyn_stat_small_proof}
\inf_{(\flow,\m)\in \CC(\base,\marg_1,\ldots,\marg_{\kk})}\actd(\flow,\m)\le \actd(\flowop,\mop)\le \acts(\jointop)=\min_{\joint\in \Pi(\marg_1,\ldots,\marg_{\kk})}\acts(\joint).
\end{equation}
Combining \eqref{eq:dyn_stat_big_proof} and \eqref{eq:dyn_stat_small_proof} we conclude that 
\begin{equation}
\label{eq:dyn_stat_equiv_inf_proof}
\inf_{(\flow,\m)\in \CC(\base,\marg_1,\ldots,\marg_{\kk})}\actd(\flow,\m)=\min_{\joint\in \Pi(\marg_1,\ldots,\marg_{\kk})}\acts(\joint),
\end{equation}
as claimed. 

Let us show that the infimum in the dynamical problem is in fact a minimum. By \eqref{eq:dyn_stat_equiv_inf_proof} we see that all the inequalities in  \eqref{eq:dyn_stat_small_proof} saturate, which means that $(\flowop,\mop)$ attains the infimum in the dynamical problem, and hence
\begin{equation}
\label{eq:dyn_stat_equiv_inf_proof_min}
\min_{(\flow,\m)\in \CC(\base,\marg_1,\ldots,\marg_{\kk})}\actd(\flow,\m)=\min_{\joint\in \Pi(\marg_1,\ldots,\marg_{\kk})}\acts(\joint).
\end{equation}
 It remains to establish items (1-2) in  Theorem \ref{thm:static_dynamic_equiv}. For item (1), since all the inequalities in \eqref{eq:dyn_stat_small_proof} saturate, we may conclude that all the inequalities in \eqref{eq:smaller_than} saturate, which shows that a minimizer $\jointop$ of the static problem induces a minimizer $(\flowop,\mop)$ of the dynamical problem. The fact  that $\flowop(1,\cdot)=\jointop$ holds by the construction of Proposition \ref{prop:velocity_exist}. 
 
For item (2), let $\{\reg(\ell)\}_{\ell}$ be a sequence of probability kernels  converging weakly to a point-mass at the origin, and given an optimal flow $(\flowop,\mop)\in\CC(\base,\marg_1,\ldots,\marg_{\kk})$  denote $\flow^{(\ell)}:=\flowop\ast \reg(\ell)$  and $\m^{(\ell)}:=\mop\ast \reg(\ell)$. Note that by construction $\flow^{(\ell)}\to \flowop$ weakly as $\ell\to\infty$. As shown in  the argument establishing Equations \eqref{eq:bigger_than_ell}-\eqref{eq:bigger_than_infty}, 
\begin{align}
\begin{split}
\label{eq:biggerEqual_than_opt}
& \min_{\joint\in \Pi(\marg_1,\ldots,\marg_{\kk})}\acts(\joint)\overset{\eqref{eq:dyn_stat_equiv_inf_proof_min} }{=}\actd(\flowop,\mop)\ge \lim_{\ell\to\infty}\int_{\domainot^{\kk}\times C([0,1];\domainot^{\kk})}\cost\left(\diffeo^{\reg(\ell)}(1,x)-x\right)\diff \mathbb P^{\reg(\ell)}( x,\diffeo^{\reg(\ell)})\\
&\ge \min_{\joint\in \Pi(\marg_1,\ldots,\marg_{\kk})}\acts(\joint).
\end{split}
\end{align}
It follows that the inequalities in \eqref{eq:biggerEqual_than_opt} saturate so that, arguing as in \eqref{eq:bigger_than_ell}-\eqref{eq:bigger_than_infty},
\begin{align}
\begin{split}
&\min_{\joint\in \Pi(\marg_1,\ldots,\marg_{\kk})}\acts(\joint)=\lim_{\ell\to\infty}\int_{\domainot^{\kk}\times C([0,1];\domainot^{\kk})}\cost\left(\diffeo^{\reg(\ell)}(1,x)-x\right)\diff\mathbb P^{\reg(\ell)}(x,\diffeo^{\reg(\ell)})\\
&=\int_{\domainot^{\kk}\times \domainot^{\kk}}\cost(y-x)\diff \joint^{\infty}(x, y),
\end{split}
\end{align}
where $\joint^{\infty}$ is a probability measure on $\domainot^{\kk}\times \domainot^{\kk}$ such that  the first marginal of $\joint^{\infty}$ on $\domainot^{\kk}$ is $\base$, and the second marginal of $\joint^{\infty}$ on $\domainot^{\kk}$ is $\flowop(1,\cdot)$. Hence,
\begin{equation}
\label{eq:pi_opt_static}
\min_{\joint\in \Pi(\marg_1,\ldots,\marg_{\kk})}\acts(\joint)=\int_{\domainot^{\kk}\times \domainot^{\kk}}\cost(y-x)\diff \joint^{\infty}(x, y)\overset{\eqref{eq:trans_inv_cost_thm}}{=}\int_{\domainot^{\kk}}\cost(y) \diff\flowop(1, y)=\acts(\flowop(1,\cdot)),
\end{equation}
as claimed.

\section{The dual dynamical multi-marginal optimal transport problem}
\label{sec:dual}
\subsection{Main result}
\label{subsec:main_result}
In this section we state the dual of our dynamical multi-marginal optimal transport problem (Definition \ref{def:dynamic_main}). We begin by recalling the dual of the static multi-marginal optimal transport problem (Definition \ref{def:static_main}).

\begin{definition}[Dual static multi-marginal  optimal transport]
\label{def:dual_static_main}
Let $\marg_1,\ldots,\marg_{\kk}$ be probability measures on a Borel set  $\domainot\subseteq \R^{\dd}$.  The \emph{dual static multi-marginal    optimal transport problem} with cost $\cost:\R^{\kk\dd}\to\R$ is
\begin{equation}
\label{eq:dual_static_mmot_def}
\sup_{\lm_1,\ldots,\lm_{\kk}}\sum_{l=1}^{\kk}\int_{\domainot}\lm_l(x_l)\diff \marg_l(x_l),
\end{equation}
where the supremum is taken over all  functions $\lm_1,\ldots,\lm_{\kk}:\domainot\to \R$, with $\lm_l\in L^1(\marg_l)$ for all $l\in [\kk]$, satisfying 
\begin{equation}
	\label{eq:dual_cost_condition}
	\sum_{l=1}^{\kk}\lm_l(x_l)\le \cost(x_1,\ldots,x_{\kk})\qquad\text{for}\quad \text{$(\otimes_{l=1}^{\kk}\marg_l)$-almost-all}\qquad (x_1,\ldots,x_{\kk})\in \domainot^{\kk}.
\end{equation}
\end{definition}

Standard duality (e.g. \cite[Theorem 1.3]{villani2021topics} for the two-marginal case whose proof extends to the multi-marginal setting) shows that the value of \eqref{eq:dual_static_mmot_def}, subject to \eqref{eq:dual_cost_condition}, coincides with the value of the static multi-marginal optimal problem (Definition \ref{def:static_main}). The following theorem is the main result of this section.

\begin{theorem}
\label{thm:dual_dynamical}
Let $\domainot\subseteq\R^{\dd}$ be a compact convex set, and let $\base$ be  a probability measure on $\domainot^{\kk}$. Let $\cost:\R^{\kk\dd}\to \R$ be a continuously differentiable cost function satisfying $\lim_{\beta\uparrow+\infty}\frac{1}{\beta}\int_{\R^{\kk\dd}}\cost(\beta x)\diff \base(x)=0$, and suppose that $\cost$ is convex and $\base$-translation-invariant. Let $\CC(\base,\marg_1,\ldots,\marg_{\kk})$ be as in \eqref{eq:cont_set}, $\actd$ as in \eqref{eq:actd_def}, and define 
\begin{equation}
\label{eq:dual_class}
\ID(\marg_1,\ldots,\marg_{\kk}):=
\begin{cases}
\lm\in C^1([0,1]\times\domainot^{\kk};\R) \\
\lm_l:\domainot\to \R \quad \textnormal{$\marg_l$-integrable for each $l\in [\kk]$}\\
\partial_t\lm(t,x)+\cost^*(\nabla \lm(t,x))\le 0,\quad \textnormal{for all}\quad (t,x)\in [0,1]\times\domainot^{\kk}\\
\sum_{l=1}^{\kk}\lm_l(x_l)\le \lm(1,x)\quad \textnormal{for}\quad \textnormal{$(\otimes_{l=1}^{\kk}\marg_l)$-almost-all.} \quad x=(x_1,\ldots,x_{\kk})\in \domainot^{\kk}
\end{cases}.
\end{equation}
Then,
\begin{equation}
\begin{split}
\label{eq:dual_thm}
&\inf_{(\flow,\m)\in \CC(\base,\marg_1,\ldots,\marg_{\kk})}\actd(\flow,\m)=\sup_{(\lm,\{\lm_l\}_{l=1}^{\kk})\in \ID(\marg_1,\ldots,\marg_{\kk})}\left\{\sum_{l=1}^{\kk}\int_{\domainot}\lm_l(x_l)\diff\marg_l( x_l) - \int_{\domainot^{\kk}}\lm(0,x)\diff \base(x)\right\}.
\end{split}
\end{equation} 
\end{theorem}

\begin{remark}[Attainment of supremum]
\label{rem:sup_attained}
We can show that, formally, the supremum on the right-hand side of \eqref{eq:dual_thm} is attained if we drop some of the regularity assumptions on $\lm$. To construct such $\lm,\{\lm_l\}_{l\in [\kk]}$ we will use the notation $\phi^{\cost}(x):=\inf_{z\in \domainot^{\kk}}\{-\phi(z)+\cost(z-x)\}$ for $\phi:\domainot^{\kk}\to \R$. Let $\{\lm_l\}_{l\in [\kk]}$ be optimizers of the dual static  multi-marginal optimal transport problem of Definition \ref{def:dual_static_main}, and define $\lm$ by 
\begin{equation}
	\label{eq:HL_formula}
	\lm(t,y):=\inf_{x\in \domainot^{\kk}}\left\{\lm(0,x)+\frac{1}{t}\cost\left(\frac{y-x}{t}\right)\right\},
\end{equation}
with $\lm(0,x):=-\left(\sum_{l=1}^{\kk}\lm_l\right)^{\cost}(x)$. Let us first show that $\lm(1,\cdot)\ge \sum_{l=1}^{\kk}\lm_l$. Indeed, given fixed $x$ we have, by definition, 
\[
\lm(0,x)=-\left(\sum_{l=1}^{\kk}\lm_l\right)^{\cost}(x)\ge \sum_{l=1}^{\kk}\lm_l(y_l)-\cost(y-x)\qquad\textnormal{for all $y$,}
\]
so, by \eqref{eq:HL_formula}, 
\[
\lm(1,y)\ge \sum_{l=1}^{\kk}\lm_l(y_l)-\cost(y-x)+\cost(y-x)= \sum_{l=1}^{\kk}\lm_l(y_l).
\]
Next, $\lm$ satisfies $\partial_t\lm(t,x)+\cost^*(\nabla \lm(t,x))\le 0$ since \eqref{eq:HL_formula} is the Hopf-Lax formula. Finally, $\lm(0,x)=0$ for $x\in\supp(\base)$, as $\cost$ is $\base$-translation-invariant and as $\{\lm_l\}_{l=1}^{\kk}$ are optimizers in the dual static multi-marginal problem, so the right-hand side of \eqref{eq:dual_thm} is equal to the left-hand side of \eqref{eq:dual_thm} by the optimality of $\{\lm_l\}_{l=1}^{\kk}$. However, $\lm$ is not a $C^1$ function in $t$ at $t=0$, so this $\lm$ is not in $\ID(\marg_1,\ldots,\marg_{\kk})$. Let us also remark that the above construction can be used to give a formal proof that the dual of the dynamical multi-marginal optimal transport problem is equivalent to the dual of the static multi-marginal optimal transport problem.
\end{remark}

\subsection{Heuristic argument for Theorem \ref{thm:dual_dynamical}}  

We begin by giving a formal proof of the dynamical duality, which also shows how one would guess the form of the dual problem to begin with. Define the Lagrangian
\begin{equation}
\label{eq:Lagrangian}
\lag(\flow,\m,\lm):=\actd(\flow,\m)+\int_0^1\int_{\domainot^{\kk}}\lm(t,x)[\partial_t\flow(t,x)+\nabla\m(t, x)]\diff x\diff t.
\end{equation}
Assuming strong duality,
\begin{align*}
\inf_{\{(\flow,\m)\,|\,\partial_t\flow+\nabla\m=0\}}\actd(\flow,\m)=\inf_{\flow,\m}\sup_{\lm}\lag(\flow,\m,\lm)=\sup_{\lm}\inf_{\flow,\m}\lag(\flow,\m,\lm),
\end{align*}
where the infimum is over $\flow$ satisfying $\flow(0,\cdot)=\base$ and $\flow_l(1,\cdot)=\marg_l$ for $l\in [\kk]$. Thus, our goal is to compute $\inf_{(\flow,\m)}\lag(\flow,\m,\lm)$ under the boundary conditions on $\flow$. By integration by parts,
\begin{equation}
\label{eq:Lagrangian_ibp_sep}
\begin{split}
\lag(\flow,\m,\lm)&=\int_0^1\int_{\domainot^{\kk}}\left[\cost\left(\frac{\m(t, x)}{ \flow(t,x)}\right) \flow(t,x)-\partial_t\lm(t,x)\flow(t,x)-\langle\nabla\lm(t,x),\m(t, x)\rangle\right]\diff x\diff t\\
&+\int_{\domainot^{\kk}}\lm(1,x)\flow(1, x)\diff x - \int_{\domainot^{\kk}}\lm(0,x) \base(x)\diff x\\
&=\int_0^1\int_{\domainot^{\kk}}\left[\cost(\vel(t,x))-\partial_t\lm(t,x)-\langle\nabla\lm(t,x),\vel(t,x)\rangle\right]\flow(t,x)\diff x\diff t\\
&+\int_{\domainot^{\kk}}\lm(1,x)\flow(1,x) \diff  x- \int_{\domainot^{\kk}}\lm(0,x) \base(x)\diff x.
\end{split}
\end{equation}
For fixed $\lm$ and $\flow$, $\lag(\flow,\m,\lm)$ is minimized by $\m^*$ where $\m^*=\vel^*\flow$ with  $\vel^*$ a minimizer of
\begin{align}
\label{eq:cost_legendre_transform}
\min_{\vel(t,x)}\left[\cost(\vel(t,x))-\langle\nabla\lm(t,x),\vel(t,x)\rangle\right]=-\cost^*(\nabla\lm(t,x)).
\end{align}
Plugging this $\m^*$ into the Lagrangian we get 
\begin{equation}
\label{eq:Lagrangian_ibp_sep_m_star}
\begin{split}
\lag(\flow,\m^*,\lm)&=\int_0^1\int_{\domainot^{\kk}}\left[-\cost^*(\nabla\lm(t,x))-\partial_t\lm(t,x)\right]\flow(t,x)\diff  x\diff t\\
&+\int_{\domainot^{\kk}}\lm(1,x)\flow(1,x) \diff  x- \int_{\domainot^{\kk}}\lm(0,x) \base( x) \diff x.
\end{split}
\end{equation}
Next we minimize $\lag(\flow,\m^*,\lm)$ over all nonnegative measures $\flow$ with marginals $\marg_1,\ldots,\marg_{\kk}$. To avoid infinity we must have
\begin{equation}
\label{eq:HJ}
\partial_t\lm(t,x)+\cost^*(\nabla\lm(t,x))\le 0.  
\end{equation}
Since the last term on the right-hand side of \eqref{eq:Lagrangian_ibp_sep_m_star} is independent of $\flow$, it follows that minimizing $\lag(\flow,\m^*,\lm)$ over $\flow$ is equivalent to minimizing  $\int_{\domainot^{\kk}}\lm(1,x)\diff\flow(1,x)$ over $\flow$ with marginals $\marg_1,\ldots,\marg_{\kk}$. By duality, this minimum is the same as $\max_{\lm_1,\ldots, \lm_{\kk}}\sum_{l=1}^{\kk}\int_{\domainot}\lm_l(x_l)\diff \marg_l(x_l)$, where $\{\lm_l\}_{l=1}^{\kk}$ satisfy  $\sum_{l=1}^{\kk}\lm(x_l)\le \lm(1,x)$. Finally, we take the maximum over all possible $\lm$ to conclude the argument.

\subsection{Proof of Theorem \ref{thm:dual_dynamical}}
We now turn to the rigorous proof of Theorem \ref{thm:dual_dynamical}. We start by showing that there is an inequality in \eqref{eq:dual_thm}, 
\begin{equation}
\label{eq:dual_thm_neq}
\inf_{(\flow,\m)\in \CC(\base,\marg_1,\ldots,\marg_{\kk})}\actd(\flow,\m)\ge\sup_{(\lm,\{\lm_l\}_{l=1}^{\kk})\in \ID(\marg_1,\ldots,\marg_{\kk})}\left\{\sum_{l=1}^{\kk}\int_{\domainot}\lm_l(x_l)\diff\marg_l( x_l) - \int_{\domainot^{\kk}}\lm(0,x)\diff \base(x)\right\}.
\end{equation} 
Fix any $(\lm,\{\lm_l\}_{l=1}^{\kk})\in \ID(\marg_1,\ldots,\marg_{\kk})$ and $(\flow,\m)\in \CC(\base,\marg_1,\ldots,\marg_{\kk})$. Then, using the definitions of $\ID(\marg_1,\ldots,\marg_{\kk})$ and $\CC(\base,\marg_1,\ldots,\marg_{\kk})$,
\begin{align*}
&\sum_{l=1}^{\kk}\int_{\domainot}\lm_l(x_l)\diff \marg_l(x_l) - \int_{\domainot^{\kk}}\lm(0,x)\diff\base(x)=\int_{\domainot^{\kk}}\sum_{l=1}^{\kk}\lm_l(x_l)\diff \flow(1,x) - \int_{\domainot^{\kk}}\lm(0,x)\diff\flow(0, x)\\
&\le \int_{\domainot^{\kk}}\lm(1,x)\diff\flow(1, x) - \int_{\domainot^{\kk}}\lm(0,x)\diff\flow(0, x)=\int_0^1\int_{\domainot^{\kk}}\partial_t\lm(t,x)\diff \flow(t,x)\diff t+\int_0^1\int_{\domainot^{\kk}}\langle\nabla\lm(t,x),\diff \m(t,x)\rangle \diff t\\
&\le \int_0^1\int_{\domainot^{\kk}}\left[\nabla\lm(t,x),\vel(t, x)\rangle-\cost^*(\nabla \lm(t,x))\right]\diff\flow(t, x)\diff t\le \int_0^1\int_{\domainot^{\kk}}\cost(\vel(t, x))\diff\flow(t,x)\diff t=\actd(\flow,\m).
\end{align*}
Taking the supremum over $(\lm,\{\lm_l\}_{l=1}^{\kk})\in \ID(\marg_1,\ldots,\marg_{\kk})$, and the infimum over $(\flow,\m)\in \CC(\base,\marg_1,\ldots,\marg_{\kk})$, yields \eqref{eq:dual_thm_neq}.

Next we construct a sequence $\{(\lm^{\epsilon},\{\lm_l^{\epsilon}\}_{l=1}^{\kk})\}_{\epsilon \ge 0}\in \ID(\marg_1,\ldots,\marg_{\kk})$ such that 
\begin{equation}
\label{eq:dual_thm_eps}
 \lim_{\epsilon\to 0}\left\{\sum_{l=1}^{\kk}\int_{\domainot}\lm_l^{\epsilon}(x_l)\diff\marg_l( x_l) - \int_{\domainot^{\kk}}\lm^{\epsilon}(0,x)\diff \base(x)\right\}=\inf_{(\flow,\m)\in \CC(\base,\marg_1,\ldots,\marg_{\kk})}\actd(\flow,\m),
\end{equation} 
which proves \eqref{eq:dual_thm}. Fix $\epsilon>0$, and let $\{\lm_l^{\epsilon}\}_{l=1}^{\kk}$ be optimal functions in the dual static multi-marginal  optimal transport problem of Definition \ref{def:dual_static_main} with cost $\R^{\kk\dd}\ni x\mapsto (1+\epsilon)\cost\left(\frac{x}{1+\epsilon}\right)$. Note that $\lm_l^{\epsilon}$ is $\marg_l$-integrable for each $l\in [\kk]$.  Let $\lm^{\epsilon}:[0,1]\times\domainot^{\kk}\to \R$ be given by $\lm^{\epsilon}(t,x):=(t+\epsilon)\cost\left(\frac{x}{t+\epsilon}\right)$. Since $\cost$ is continuously differentiable we have that $\lm^{\epsilon}\in  C^1([0,1]\times\domainot^{\kk};\R)$. In addition, since $\{\lm_l^{\epsilon}\}_{l=1}^{\kk}$  are optimizers for the dual static multi-marginal optimal transport problem with cost $(1+\epsilon)\cost\left(\frac{x}{1+\epsilon}\right)$, we have
\begin{equation}
\label{eq:lm_eps_smaller_lm}
\sum_{l=1}^{\kk}\lm_l^{\epsilon}(x_l)\le (1+\epsilon)\cost\left(\frac{x}{1+\epsilon}\right)=\lm^{\epsilon}(1,x)\qquad\text{for}\quad \text{$(\otimes_{l=1}^{\kk}\marg_l)$-almost-all}\qquad (x_1,\ldots,x_{\kk})\in \domainot^{\kk}.
\end{equation}
Hence, to conclude that $\{(\lm^{\epsilon},\{\lm_l^{\epsilon}\}_{l=1}^{\kk})\}_{\epsilon \ge 0}\in \ID(\marg_1,\ldots,\marg_{\kk})$ it remains to show that $\partial_t\lm^{\epsilon}+\cost^*(\nabla\lm^{\epsilon})\le 0$. Indeed,
\begin{align}
\label{eq:HJ_eps}
\partial_t\lm^{\epsilon}(t,x)=\cost\left(\frac{x}{t+\epsilon}\right)-\left\langle \nabla \cost\left(\frac{x}{t+\epsilon}\right),\frac{x}{t+\epsilon}\right\rangle=-\cost^*\left(\nabla\cost \left(\frac{x}{t+\epsilon}\right)\right)=-\cost^*(\nabla\lm^{\epsilon} (t,x)).
\end{align}
Next we establish \eqref{eq:dual_thm_eps}. Since $\{\lm_l^{\epsilon}\}_{l=1}^{\kk}$  are optimizers for the dual  static multi-marginal problem with cost $(1+\epsilon)\cost\left(\frac{x}{1+\epsilon}\right)$, duality gives
\begin{align*}
\sum_{l=1}^{\kk}\int_{\domainot}\lm_l^{\epsilon}(x_l)\diff\marg_l( x_l) = \inf_{\joint \in \Pi(\marg_1,\ldots,\marg_{\kk})}\acts^{\epsilon}(\joint),
\end{align*}
where $\acts^{\epsilon}(\joint):=\int_{\domainot^{\kk}}(1+\epsilon)\cost\left(\frac{x}{1+\epsilon}\right)\diff \joint( x)$. Hence, by the definition of $\lm^{\epsilon}$, 
\begin{align*}
&\sum_{l=1}^{\kk}\int_{\domainot}\lm_l^{\epsilon}(x_l)\diff\marg_l( x_l) - \int_{\domainot^{\kk}}\lm^{\epsilon}(0,x)\diff\base( x)=\inf_{\joint \in \Pi(\marg_1,\ldots,\marg_{\kk})}\acts^{\epsilon}(\joint)- \int_{\domainot^{\kk}}\epsilon\cost\left(\frac{x}{\epsilon}\right)\diff \base(x).
\end{align*}
By assumption, $\lim_{\epsilon\to 0}\int_{\domainot^{\kk}}\epsilon\cost\left(\frac{x}{\epsilon}\right)\diff \base(x)=0$, so using \eqref{eq:static_dynamic_equiv}, we will establish \eqref{eq:dual_thm_eps} once the following claim is verified (see \cite[Theorem 5.20]{villani2008optimal} for an alternative argument in the two-marginal case). 
\begin{claim}
\[
\lim_{\epsilon\to 0}\inf_{\joint \in \Pi(\marg_1,\ldots,\marg_{\kk})}\acts^{\epsilon}(\joint)=\inf_{\joint \in \Pi(\marg_1,\ldots,\marg_{\kk})}\acts(\joint).
\]
\end{claim}
\begin{proof}
We begin by showing that if $\{\joint^{\epsilon}\}$ is a sequence in $\Pi(\marg_1,\ldots,\marg_{\kk})$ which converges to some $\joint\in\Pi(\marg_1,\ldots,\marg_{\kk})$, then $\lim_{\epsilon\to 0}\acts^{\epsilon}(\joint^{\epsilon})=\acts(\joint)$. Indeed, denoting $\cost^{\epsilon}(x):=(1+\epsilon)\cost\left(\frac{x}{1+\epsilon}\right)$, we have
\begin{align*}
\acts^{\epsilon}(\joint^{\epsilon})=\int_{\domainot^{\kk}}\cost^{\epsilon}(x) \diff\joint^{\epsilon}( x)&=\int_{\domainot^{\kk}}\cost(x) \diff\joint( x)\\
&+\int_{\domainot^{\kk}}[\cost^{\epsilon}(x)-\cost(x)]\diff \joint^{\epsilon}(x)+\left[\int_{\domainot^{\kk}}\cost(x)\diff\joint^{\epsilon}(x)-\int_{\domainot^{\kk}}\cost(x) \diff \joint(x)\right].
\end{align*}
Let us show that the last two terms converge to zero as $\epsilon\to 0$. For the first term we have
\begin{align*} 
&\left|\int_{\domainot^{\kk}}[\cost^{\epsilon}(x)-\cost(x)]\diff\joint^{\epsilon}(  x)\right|\le\int_{\domainot^{\kk}}\left|\cost\left(\frac{x}{1+\epsilon}\right)-\cost(x)\right|\diff \joint^{\epsilon}( x)+\epsilon \left|\int_{\domainot^{\kk}}\cost\left(\frac{x}{1+\epsilon}\right)\diff \joint^{\epsilon}( x) \right|\\
&\le\sup_{y\in \domainot_{\epsilon}^{\kk}}|\nabla \cost(y)|\frac{\epsilon}{1+\epsilon}\int_{\domainot^{\kk}}|x|\diff \joint^{\epsilon}( x)+\epsilon\sup_{y\in \domainot_{\epsilon}^{\kk}}\cost(y),
\end{align*}
where $\domainot^{\kk}_{\epsilon}:=\cup_{r\in[\frac{1}{1+\epsilon},1]}\,r\domainot^{\kk}$. The terms $\sup_{y\in \domainot_{\epsilon}^{\kk}}|\nabla \cost(y)|$, $\int_{\domainot^{\kk}}|x| \diff \joint^{\epsilon}(x)$, and $\sup_{y\in \domainot_{\epsilon}^{\kk}}\cost(y)$ are finite since $\domainot_{\epsilon}^{\kk}$ is compact, and as $\cost$ is continuously differentiable. Hence, $|\int_{\domainot^{\kk}}[\cost^{\epsilon}(x)-\cost(x)] \diff\joint^{\epsilon}( x)|\to 0$ as $\epsilon\to 0$.  The second term also converges to 0, $[\int_{\domainot^{\kk}}\cost(x)\diff \joint^{\epsilon}(x)-\int_{\domainot^{\kk}}\cost(x)\diff\joint(  x)]\to 0$ as  $\epsilon\to 0$, since $\joint^{\epsilon}\to \joint$ weakly, and $\cost$ is continuous and bounded (as $\domainot^{\kk}$ is compact). To conclude, we have shown that $\lim_{\epsilon\to 0}\acts^{\epsilon}(\joint^{\epsilon})=\acts(\joint)$ if $\lim_{\epsilon\to 0}\joint^{\epsilon}=\joint$ weakly. 

For each $\epsilon$ let $\joint^{\epsilon}\in\argmin_{\joint \in \Pi(\marg_1,\ldots,\marg_{\kk})}\acts^{\epsilon}(\joint)$, and note that, by the compactness of $\domainot^{\kk}$,  we have a subsequence, still denoted $\{\joint^{\epsilon}\}$, which converges weakly to some $\joint \in  \Pi(\marg_1,\ldots,\marg_{\kk})$. Using the above result it suffices to show that $\joint\in \argmin_{\tilde\joint \in \Pi(\marg_1,\ldots,\marg_{\kk})}\acts(\tilde\joint)$, since we then have 
\begin{equation}
\label{eq:claim_lim}
\lim_{\epsilon\to 0}\inf_{\joint \in \Pi(\marg_1,\ldots,\marg_{\kk})}\acts^{\epsilon}(\joint)=\lim_{\epsilon\to 0}\acts^{\epsilon}(\joint^{\epsilon})=\acts(\joint)=\inf_{\tilde\joint \in \Pi(\marg_1,\ldots,\marg_{\kk})}\acts(\tilde\joint),
\end{equation}
establishing the claim. To verify the last equality in \eqref{eq:claim_lim} note that by definition
\begin{equation}
\label{eq:claim_inq}
\acts^{\epsilon}(\joint^{\epsilon})\le \acts^{\epsilon}(\tilde\joint)\qquad\text{for all}\qquad \tilde\joint\in \Pi(\marg_1,\ldots,\marg_{\kk}).
\end{equation}
We have shown that the left-hand side of \eqref{eq:claim_inq} converges to $\acts(\joint)$, and, by the same argument, the right-hand side of \eqref{eq:claim_inq} converges to $ \acts(\tilde\joint)$. In other words, 
\begin{equation}
\label{eq:claim_inq2}
\acts(\joint)\le \acts(\tilde\joint)\qquad\text{for all}\qquad \tilde\joint\in \Pi(\marg_1,\ldots,\marg_{\kk}),
\end{equation}
which shows that $\joint\in \argmin_{\tilde\joint \in \Pi(\marg_1,\ldots,\marg_{\kk})}\acts(\tilde\joint)$.
\end{proof}

\section{Proximal splitting methods}
\label{sec:proximal}
\subsection{The discretized dynamical multi-marginal optimal transport problem}
\label{subsec:disc_mmot}
In order to numerically solve the dynamical multi-marginal optimal transport problem we resort to discretization. We will use the specific discretization of \cite{baradat2021regularized}, which in turns follows  \cite{papadakis2014optimal}, to obtain a finite-dimensional convex optimization problem. The objective of this optimization problem is convex, but non-smooth, which excludes methods such as gradient descent. Following \cite{benamou2000computational, papadakis2014optimal, baradat2021regularized} we will use proximal splitting, specifically  the primal-dual method  \cite[\S 4.5]{papadakis2014optimal} (but other proximal splitting methods are applicable as well). This is not the only approach, of course, and other methods can be found in \cite{papadakis2014optimal}.

To simplify the exposition we will restrict to $\dd=1$, but the method we describe in this section, and our numerical implementation, generalize to arbitrary dimensions (but can be painfully slow as the dimension grows). Another reason to focus on $\dd=1$ is that in this setting we can compare our numerical solutions to the known analytic ones; see Section \ref{subsec:numerics}. Note that although $\dd=1$, the overall dimension will be $\dd\kk=\kk$ where $\kk$ is the number of marginals. As a concrete example we take the domain of the marginals to be the one-dimensional torus\footnote{Our theoretical results were formulated for domains with boundary. The choice to use a periodic domain in the discretized problem is simply for convenience and should not be essential.} $\domainot = \mathbb T$, i.e., the interval $[0,1]$  with periodic boundary conditions. We denote by $N_t$ and $N_x$ the number of discretization points of $[0,1]$ and $\mathbb T$, respectively. As common in discretizations of the continuity equation, we use centered and staggered grids for time and space. We begin with the one-dimensional grids. The temporal centered and staggered  grids are
\begin{equation}
\label{eq:gridt_def}
\gridtc:=\left\{\frac{i}{N_t}+\frac{1}{2 N_t}~: ~0\le i\le N_t-1\right\}\subset [0,1], \qquad \gridts:=\left\{\frac{i}{N_t}~: ~0\le i\le N_t\right\}\subset[0,1],
\end{equation}
respectively, while the spatial centered and staggered  grids are
\begin{equation}
\label{eq:gridx_def}
\gridxc:=\left\{\frac{j}{N_x}~: ~0\le j\le N_x-1\right\}\subset \mathbb T, \qquad \gridxs:=\left\{\frac{j}{N_x}+\frac{1}{2N_x}~: ~0\le j\le N_x-1\right\}\subset \mathbb T,
\end{equation}
respectively. We will have two sets of variables, \emph{centered variables}, which  live on the centered grids, and \emph{staggered variables}, which live on a mix of the centered and staggered grids. The set of centered variables is defined as
\begin{equation}
\label{eq:fgridc_def}
\fgridc :=\left\{\Uc := (\fc, \mc)\in \R^{\gridtc\times (\gridxc)^{\kk}}\times (\R^{\kk})^{\gridtc\times (\gridxc)^{\kk}}\right\},
\end{equation}
where $\fc, \mc$ stand for the temporal and spatial centered discretizations of $\pi,\m$, and the values of $\fc, \mc$ at a point $(t,x)\in \gridtc\times (\gridxc)^{\kk}$,  denoted  $\fc(t,x)$ and $\mc(t,x)$, stand for the total amount of mass of a box centered at $(t,x)$. The objective function of our discretized dynamical multi-marginal optimal transport problem is defined on the centered variables as $\actdd:\fgridc\to \R\cup\{+\infty\}$ given by
\begin{equation}
\label{eq:obj}
\actdd(\Uc)=\actdd(\fc,\mc):=\sum_{(t,x)\in\gridtc\times (\gridxc)^{\kk}}\cost\left(\frac{\mc(t,x)}{\fc(t,x)}\right)\fc(t,x).
\end{equation}
The discretization of the continuity equation constraint is done on the set of staggered variables, which is defined as
\begin{equation}
\label{eq:fgrids_def}
\fgrids :=\left\{\Us := (\fs, \ms)\in \R^{\gridts\times (\gridxc)^{\kk}}\times (\R^{\kk})^{\gridtc\times (\gridxs)^{\kk}}\right\},
\end{equation}
where $\fs, \ms$ play the analogous roles of $\fc,\mc$, respectively. 
The continuity equation constraint becomes, for all $(t,x)\in \gridtc\times (\gridxc)^{\kk}$,
\begin{equation}
\label{eq:cont_eq_disc}
\frac{1}{N_t}\left[\fs\left(t+\frac{1}{2N_t}, x\right)-\fs\left(t-\frac{1}{2N_t}, x\right)\right]+\frac{1}{N_x}\sum_{l=1}^{\kk}\left[\ms\left(t,x+\frac{e_l}{2N_x}\right)-\ms\left(t, x-\frac{e_l}{2N_x}\right)\right]=0,
\end{equation}
where $\{e_l\}_{l=1}^{\kk}$ is the standard basis of $\R^{\kk}$. We denote the set of staggered variables satisfying the continuity equation as
\begin{equation}
\label{eq:cont_eq_disc_def}
\mathcal C\fgrids:=\{\Us\in \fgrids: \Us\textnormal{ satisfies \eqref{eq:cont_eq_disc} for all $(t,x)\in \gridtc\times (\gridxc)^{\kk}$}\}.
\end{equation}
In addition, we need to impose the boundary conditions so we set 
\begin{equation}
\label{eq:cont_eq_disc_plus_bd_def}
\mathcal C\fgrids(\marg_1,\ldots,\marg_{\kk}):=\left\{\Us\in \mathcal C\fgrids ~: \sum_{x_{\backslash l}\in (\gridxc)^{\kk-1}}\fs(1,(x_l, x_{\backslash l}))=\marg_l(x_l)\quad\textnormal{for all }x_l\in \gridxc\quad\textnormal{for all } l\in [\kk]\right\},
\end{equation}
where we write $x=(x_l, x_{\backslash l})$ for $x\in (\gridxc)^{\kk}$, with $x_l\in \gridxc$ standing for the $l$th variable, and $x_{\backslash l}\in (\gridxc)^{\kk-1}$ standing for the rest of the variables. 

Finally, the  sets of centered and staggered variables can be mapped onto each other via $\mathcal I:\fgrids\to \fgridc$ given by
\begin{equation}
\label{eq:avg_operator}
\mathcal I\begin{pmatrix}\fs \\ \ms\end{pmatrix}(t,x)=\frac{1}{2}\begin{pmatrix}\fs(t+\frac{1}{2N_t}, x)+\fs(t-\frac{1}{2N_t}, x)\\ \sum_{l=1}^{\kk}\left[\ms(t, x+\frac{e_l}{2N_x})+\ms(t, x-\frac{e_l}{2N_x})\right]\end{pmatrix},
\end{equation} 
where $\{e_l\}_{l=1}^{\kk}$ is again the standard basis of $\R^{\kk}$.

Putting everything together we get the discretized version of Definition \ref{def:dynamic_main}.
\begin{definition}[Discretized dynamical multi-marginal optimal transport problem]
\label{def:disc_DMMOT}
Let $\marg_1,\ldots\marg_{\kk}$ be probability measures on  $\gridxc$. The \emph{discretized dynamical multi-marginal optimal transport problem} is
\begin{align}
\begin{split}
\label{eq:disc_DMMOT}
&\min_{\Uc\in \fgridc} \,\actdd(\Uc)\quad\textnormal{subject to} \quad \Uc= \mathcal I (\Us)\quad\textnormal{where}\quad \Us\in \mathcal C\fgrids(\marg_1,\ldots\marg_{\kk}).
\end{split}
\end{align}

\end{definition}
As is clear from Definition \ref{def:disc_DMMOT}, the discretized dynamical multi-marginal optimal transport problem is a convex optimization problem, where the objective is convex and the constraints are linear. 

\subsection{The primal-dual proximal splitting method}

On a very high level, a proximal splitting approach towards convex optimization problems (such as the one in Definition \ref{def:disc_DMMOT}) alternates between steps which minimize the objective, and steps which enforce the constraints. The ``splitting" part refers to the various ways to perform these alternate steps, and the ``proximal" part refers to the method by which the objective is minimized and the constraints are enforced. Specifically, since in our setting the objective is non-smooth (at the origin), standard gradient descent methods are problematic, which necessitates the usage of proximal operators. In general, given a finite dimensional Hilbert space $\mathcal H$, and a convex lower semi-continuous  proper function $\mathrm F:\mathcal H\to \R\cup\{+\infty\}$, the \emph{proximal operator} of $\mathrm F$ is a function $\Prox_{\mathrm F}:\mathcal H\to \R$ given by
\begin{equation}
\label{eq:prox_def}
\Prox_{\mathrm F} (h):=\argmin_{\tilde h}\left\{\frac{1}{2}\|h-\tilde h\|_{\mathcal H}^2+\mathrm F(\tilde h)\right\}.
\end{equation}
One can check that, when $\mathrm F$ is differentiable, $\Prox_{\mathrm F} (h)$ returns an implicit Euler gradient descent step on $\mathrm F$ starting at $h$. The power of proximal operators is that $\Prox_{\mathrm F}$ is well-defined with no smoothness assumptions on $\mathrm F$, and can thus serve as a substitute for gradient descent. We will need the proximal operator of the Legendre-Fenchel transform of $\mathrm F$,
\begin{equation}
\label{eq:LF_def}
\mathrm F^*(g)=\max_{h\in\mathcal H}\left\{\langle h,g\rangle_{\mathcal H}-\mathrm F(h)\right\},
\end{equation}
which is given by Moreau’s identity, for every $\alpha>0$,
\begin{equation}
\label{eq:Moreau}
 \Prox_{\alpha \mathrm  F^*}(g)=g-\alpha  \Prox_{\frac{\mathrm F}{\alpha}}\left(\frac{g}{\alpha}\right) \quad\textnormal{for all}\quad  g\in\mathcal H.
\end{equation}
The \emph{primal-dual proximal splitting method} applies to problems of the form 
\begin{equation}
\label{eq:PD_problems}
\min_{h\in \mathcal H} \left\{\mathrm F(\mathcal K(h))+\mathrm G(h)\right\},
\end{equation}
where $\mathrm F:\tilde{\mathcal H}\to \R\cup\{+\infty\}$, $\mathrm G:\mathcal H\to \R\cup\{+\infty\}$ are convex lower semi-continuous  proper functions on finite-dimensional Hilbert spaces $\tilde{\mathcal H},\mathcal H$, and $\mathcal K:\mathcal H\to\tilde{\mathcal H}$ is a linear operator. Given positive parameters $\theta,\sigma,\tau$, and initialization $f^{(0)}\in \mathcal H$ and $g^{(0)}\in \tilde{\mathcal H}$, we solve for $\ell=0,1,\ldots$,
\begin{align}
\begin{split}
\label{eq:PD_iter}
&g^{(\ell+1)}=\Prox_{\sigma\mathrm F^*}(g^{(\ell)}+\sigma \mathcal K(f^{\ell})),\\
&h^{(\ell+1)}=\Prox_{\tau\mathrm G}(h^{(\ell)}-\tau\mathcal K^*(g^{(\ell+1)})),\\
&f^{(\ell+1)}=h^{(\ell+1)}+\theta(h^{(\ell+1)}-h^{(\ell)}).
\end{split}
\end{align}
It is known \cite{chambolle2011first}, \cite[\S 4.5]{papadakis2014optimal}  that if $0\le \theta\le 1$, and $\sigma\tau \|\mathcal K\|_{\textnormal{op}}<1$, then $h^{(\ell)}\to h^*$ where $h^*$ is a minimizer of \eqref{eq:PD_problems}. 

To apply the primal-dual method in our setting we take 
\begin{equation}
\label{eq:PD_DMMOT}
\tilde{\mathcal H}=\fgridc,\qquad \mathcal H=\fgrids,\qquad \mathrm F:=\actdd,\qquad \mathcal K:=\mathcal I,\qquad \mathrm G:=\iota_{\mathcal C\fgrids(\marg_1,\ldots\marg_{\kk})},
\end{equation}
where $\iota_{\mathcal C\fgrids(\marg_1,\ldots\marg_{\kk})}$ is the $0~/+\infty$ indicator function  of the set $\mathcal C\fgrids(\marg_1,\ldots\marg_{\kk})$.  In order to implement the primal-dual method we need to be able to compute the proximal operators of $\mathrm F$ and $\mathrm G$, which is problem-dependent. In the next section    we show how to do so for the quadratic cost using a slight modification of the formulation \eqref{eq:PD_DMMOT}.

\subsection{The quadratic  cost}
\label{subsec:GS}

Fix marginals $\marg_1,\ldots,\marg_{\kk}$ on $\domainot\subseteq \R^{\dd}$. The static multi-marginal optimal transport problem with the quadratic  cost  is defined as 
\begin{equation}
\label{eq:GS_mmot_cts}
\min_{\joint\in\Pi(\marg_1,\ldots,\marg_{\kk})}\int_{\domainot^{\kk}}\sum_{i,j=1}^{\kk}|x_i-x_j|^2\diff \joint(x_1,\ldots, x_{\kk}),
\end{equation}
where $x_l\in \R^{\dd}$ for $l\in [\kk]$, and $|\cdot|$ is the Euclidean norm on $\R^{\dd}$. The dynamical formulation is thus 
\begin{equation}
\label{eq:GS_mmot_cts}
\min_{(\flow,\m)\in \CC(\base,\marg_1,\ldots,\marg_{\kk})}\int_0^1\int_{\domainot^{\kk}}\sum_{i,j=1}^{\kk}\left|\frac{\diff \m_i(t,x)-\diff\m_j(t, x)}{\diff \flow(t, x)}\right|^2\diff \flow(t,x)\diff t,
\end{equation}
where $\m=(\m_1,\ldots,\m_{\kk})$. After discretizing \eqref{eq:GS_mmot_cts} as in Section \ref{subsec:disc_mmot}, we get the optimization problem 
\begin{equation}
\label{eq:GS_dmmot_cts}
\min_{\Uc\in \fgridc} \,\sum_{(t,x)\in\gridtc\times (\gridxc)^{\kk}}\left|\frac{\mc_i(t,x)-\mc_j(t,x)}{\fc(t,x)}\right|^2\fc(t,x)\quad\textnormal{s.t.} \quad \Uc= \mathcal I (\Us)\quad\textnormal{where}\quad \Us\in \mathcal C\fgrids(\marg_1,\ldots\marg_{\kk}),
\end{equation}
with  $\mc=(\mc_1,\ldots,\mc_{\kk})$. To put \eqref{eq:GS_dmmot_cts} into a form amenable to the primal-dual method let us define the following functions:
\begin{equation}
\label{eq:GS_def}
\mathcal S: (\R^{\kk})^{\gridtc\times (\gridxc)^{\kk}}\to  (\R^{\frac{\kk(\kk+1)}{2}})^{\gridtc\times (\gridxc)^{\kk}}\qquad\textnormal{by} \qquad \mathcal S(\m):=\left(\m_i-\m_j\right)_{1\le i<j\le \kk},
\end{equation}
and
\begin{equation}
\label{eq:GS_def_complete}
\bar{\mathcal S}: \R^{\gridtc\times (\gridxc)^{\kk}}\times (\R^{\kk})^{\gridtc\times (\gridxc)^{\kk}}\to \R^{\gridtc\times (\gridxc)^{\kk}}\times (\R^{\frac{\kk(\kk+1)}{2}})^{\gridtc\times (\gridxc)^{\kk}}\qquad \textnormal{by} \qquad \bar{\mathcal S}(\flow,\m):=(\flow, \mathcal S(\m)).
\end{equation}
We define $J: \R^{\gridtc\times (\gridxc)^{\kk}}\times (\R^{\frac{\kk(\kk+1)}{2}})^{\gridtc\times (\gridxc)^{\kk}}$ by
\begin{equation}
\label{eq:J_def}
\mathrm J(\flow,\m):=\begin{cases}
\frac{|\m|^2}{\pi} & \textnormal{if } \flow>0,\\
(0,0)  & \textnormal{if } \flow=0,\\
+\infty & \textnormal{otherwise}
\end{cases}.
\end{equation}
With the above notation we set
\begin{align}
\begin{split}
\label{eq:PD_DMMOT_GS}
\tilde{\mathcal H}=\R^{\gridtc\times (\gridxc)^{\kk}}\times (\R^{\frac{\kk(\kk+1)}{2}})^{\gridtc\times (\gridxc)^{\kk}},\quad \mathcal H=\fgrids,\quad \mathrm F:=\mathrm J,\quad \mathcal K:=\bar{\mathcal S}\circ \mathcal I,\quad \mathrm G:=\iota_{\mathcal C\fgrids(\marg_1,\ldots\marg_{\kk})},
\end{split}
\end{align}
so that \eqref{eq:GS_dmmot_cts} is equivalent to \eqref{eq:PD_DMMOT_GS}, which in turn is of the form \eqref{eq:PD_problems}. In order to apply the primal-dual method \eqref{eq:PD_iter} we need to be able to compute $\Prox_{J^*}$ and $\Prox_{\iota_{\mathcal C\fgrids(\marg_1,\ldots\marg_{\kk})}}$. In fact, due to  Moreau’s identity \eqref{eq:Moreau}, it suffices to compute $\Prox_{J}$ and $\Prox_{\iota_{\mathcal C\fgrids(\marg_1,\ldots\marg_{\kk})}}$. The proximal operator  of $J$ was computed in 
\cite[Proposition 1]{papadakis2014optimal}; see also \cite[
Lemma 6.71]{baradat2021regularized}. As for the proximal operator of $\iota_{\mathcal C\fgrids(\marg_1,\ldots\marg_{\kk})}$, since the set $\mathcal C\fgrids(\marg_1,\ldots\marg_{\kk})$ is convex, we have $\Prox_{\iota_{\mathcal C\fgrids(\marg_1,\ldots\marg_{\kk})}}=\Proj_{i\mathcal C\fgrids(\marg_1,\ldots\marg_{\kk})}$, i.e., the proximal operator of $\iota_{\mathcal C\fgrids(\marg_1,\ldots\marg_{\kk})}$ is the orthogonal projection onto  $\mathcal C\fgrids(\marg_1,\ldots\marg_{\kk})$. The set $\mathcal C\fgrids(\marg_1,\ldots\marg_{\kk})$ forms a subspace so the orthogonal projection onto it can be computed by solving a system of linear equations; see \cite[
\S 6.3.4]{baradat2021regularized}.

\subsubsection{Numerical experiments}
\label{subsec:numerics}
In our numerical experiments we focus on $\dd=1$. In dimension one the \emph{analytic} solution to the static  multi-marginal optimal transport for the  quadratic cost \eqref{eq:GS_mmot_cts} is known, so we can compare our numerical results to ground truth. Specifically, it follows from \cite[Theorem 2.1]{gangbo1998optimal} that the optimal solution to the static  multi-marginal optimal transport for the  quadratic cost (in any dimension $\dd$) is unique, and is of the form

\begin{equation}
\label{eq:optimal_coupling_GS}
\jointop=\law(X_1, \mapot^{1\to 2}(X_1),\ldots, \mapot^{1\to \kk}(X_1)),\qquad\qquad X_1\sim \marg_1,
\end{equation}
where $\{\mapot^{1\to l}\}_{l=2}^{\kk}$ are gradients of convex functions. When $\dd=1$ this means that $\{\mapot^{1\to l}\}_{l=2}^{\kk}$ are monotonic, and since they map $\marg_1$ to $\{\marg_l\}_{l=2}^{\kk}$, respectively, it follows that, when the cumulative distribution functions of the marginals are invertible, we must have
\begin{equation}
\label{eq:analytic_transport_GS}
\mapot^{1\to l}=F_l^{-1}\circ F_1,\qquad l=2,\ldots, \kk, 
\end{equation}
with $F_l$ the cumulative distribution function of $\marg_l$ for $l\in [\kk]$. Thus, our goal is to recover the transport maps $\{\mapot^{1\to l}\}_{l=2}^{\kk}$. 

The output of the primal-dual proximal splitting method is a temporal-spatial discretization of a flow $\flow$ (as well as momentum), which, if optimal,  satisfies, by Theorem \ref{thm:static_dynamic_equiv}(2),
\begin{equation}
\label{eq:what_we_want}
\flow(1,\cdot)=\jointop. 
\end{equation}
In turn,  the transport maps $\{\mapot^{1\to l}\}_{l=2}^{\kk}$ can be obtained from $\jointop$ by taking the conditional expectation \cite[\S 4]{chen2014numerical},
\begin{equation}
\label{eq:approx_transport_analytic}
\mapot^{1\to l}(x_1)=\mathbb E_{X\sim\jointop}[X_l|X_1=x_1].
\end{equation}
Therefore, we compare the transport maps computed via conditional expectation with respect to a discretization of $\flow(1,\cdot)$, against their analytic form \eqref{eq:analytic_transport_GS}; see below for the precise computation.

Let us now describe our numerical experiment. We take $\dd=1$, $\kk=3$, and $\domainot = \mathbb T$, i.e., $[0,1]$ with periodic boundary conditions. 
We take the the source measure $\base$ to be the uniform measure on $[0,1]$, and we take the following marginals (cf. Figure \ref{fig:3marginals}),
\begin{equation}
\label{eq:margs_numeric}
\marg_1 (x)= 
\begin{cases}
0 & \textnormal{if } x<0\\
\frac{\frac{\pi}{2}\sin(\pi x)+\delta}{1+\delta} & \textnormal{if } 0\le x\le 1\\
0 & \textnormal{if } x> 1\
\end{cases},\quad \marg_2 (x)= 
\begin{cases}
0 & \textnormal{if } x< 0\\
4x & \textnormal{if } 0\le x<\frac{1}{2}\\
4(1-x) & \textnormal{if } \frac{1}{2}\le x\le 1\\
0 & \textnormal{if } x> 1\
\end{cases},\quad \marg_3 (x)= 
\begin{cases}
0 & \textnormal{if } x< 0\\
8x & \textnormal{if } 0\le x<\frac{1}{4}\\
4-8x & \textnormal{if } \frac{1}{4}\le x< \frac{1}{2}\\
8x-4 & \textnormal{if } \frac{1}{2}\le x< \frac{3}{4}\\
8-8x & \textnormal{if } \frac{3}{4}\le x\le 1\\
0 & \textnormal{if } x> 1\
\end{cases},
\end{equation}
with $\delta:=0.2$ (we explain the reason for the $\delta$ term below). 

Given an output $(\fs,\ms)$ of the primal-method algorithm, and an index $l\in \{2,3\}$, we marginalize $\{\fs(1,x_1,x_2,x_3)\}_{x_1,x_2,x_3\in \gridxc}$ over the third index  to get the joint laws of $(X_1,X_2)$ and $(X_1,X_3)$, denoted respectively as $\{\fs_{12}(1,x_1,x_2)\}_{x_1,x_2\in \gridxc}$ and $\{\fs_{13}(1,x_1,x_3)\}_{x_1,x_3\in \gridxc}$. With a slight abuse of notation we denote by $\marg_1$ the discretization of the corresponding distribution in \eqref{eq:margs_numeric}, after normalizing it to sum up to 1. The approximation of the transport maps is then computed as
\begin{align}
\begin{split}
\label{eq:transport_approx}
&\mathrm T^{1\to 2}_{\textnormal{approx}} (x_1) :=\frac{1}{2\pi}\text{arg}\left( \sum_{x_2\in \gridxc} \exp(2\pi \mathrm{i}\, x_2) \frac{\fs_{12}(1,x_1,x_2)}{\marg_1(x_1)}\right),\\
&\mathrm T^{1\to 3}_{\textnormal{approx}}(x_1)  :=\frac{1}{2\pi}\text{arg}\left( \sum_{x_3\in \gridxc} \exp(2\pi \mathrm{i}\, x_3) \frac{\fs_{13}(1,x_1,x_3)}{\marg_1(x_1)}\right),
\end{split}
\end{align}
where $\mathrm i$ is the imaginary unit, and arg is angle of a complex number in radians. The reason for the computation in \eqref{eq:transport_approx} is that we map $[0,1]$ to the torus $\mathbb T$ by $r\mapsto e^{2\pi \mathrm{i} r}$, where we perform the conditional expectation, and then map back to $[0,1]$. The reason why we chose $\delta>0$ in the definition of $\marg_1$ in  \eqref{eq:margs_numeric} is to avoid division by zero in \eqref{eq:transport_approx}.

We take $N_t=N_x=10$ and run the algorithm for 5000 iterations, which takes a few minutes on a standard laptop. We take $\theta=1$ (but the algorithms works for other values of $\theta$), $\sigma=85$, and $\tau=0.1$. We observe that as iterations progress, $\fs$ becomes more negative, though refining the temporal grid seems to alleviate this issue somewhat. The code\footnote{Our code builds on the code of Hugo Lavenant  \url{https://github.com/HugoLav/RegUnOT}.} to run the primal-dual algorithm is available at 

\begin{center}
\url{https://github.com/yairshenfeld/DMMOT}
\end{center}

\begin{remark}[Entropic regularization]
The  Benamou-Brenier  formulation of the dynamical two-marginal optimal transport problem can be entropically regularized by adding a Fisher information term to the objective, or, equivalently, adding a Laplacian term to the continuity equation; cf. \cite[Problem 4.6]{chen2021stochastic}. The addition of the Laplacian term to the continuity equation can also be done for the discretized problem \cite[
\S 6.3]{baradat2021regularized}. The latter regularization applies in our setting as well, and our code provides its implementation. 
\end{remark}

\bibliographystyle{alpha}
\bibliography{ref_DMMOT}

\end{document}